\documentclass[preprint,12pt,nopreprintline]{elsarticle}
\usepackage[english]{babel}
\usepackage[utf8]{inputenc}
\usepackage[OT1]{fontenc}
\usepackage{amsfonts, amsmath, amssymb, amsthm}
\usepackage{enumitem}
\usepackage{mathtools}
\usepackage{graphicx}
\usepackage{listings}
\usepackage{xcolor}
\usepackage{soul}
\usepackage{graphicx}
\usepackage{textcomp}
\usepackage{subcaption}
\usepackage[framemethod=tikz]{mdframed}

\theoremstyle{plain}
\newtheorem{theorem}{Theorem}[section]
\newtheorem{lemma}[theorem]{Lemma}
\theoremstyle{remark}
\newtheorem*{remark}{Remark}
\newcommand{\ebound}{%
  \Delta(\sigma,n,\delta')%
}
\newcommand{\nl}{%
  \widetilde{\lambda}%
}
\newcommand{\nlrp}{%
  \widetilde{\lambda}_{+}%
}
\newcommand{\nlrpi}[1]{%
  (\widetilde{\lambda}_{#1})_{+}%
}

\newcommand{\nlrpr}{%
  \widetilde{\lambda}_{+}^{\frac{1}{N}}%
}
\newcommand{\lpnr}{
\lambda_{+}^{\frac{1}{N}}
}
\newcommand{\Id}{{\operatorname{Id}}}

\newcommand{\fnorm}[1]{\left\|#1\right\|_{F}}
\newcommand{\lnorm}[1]{\left\|#1\right\|}
\newcommand{\abs}[1]{\left| #1\right|}
\newcommand{\Wk}{W(k)}
\newcommand{\hW}{\widehat{W}}

\newcommand{\wt}[1]{\widetilde{#1}}

\newcommand{\sbrac}[1]{\left(#1\right)}

\newcommand{\dtil}[1]{d_{\widetilde{\lambda}_{#1}}^N(k)}

\def\l{l}
\def\L{L}
\def\W{W}
\def\Wstar{\widehat{W}}

\DeclarePairedDelimiter\ceil{\lceil}{\rceil}

\begin{document}

\begin{frontmatter}

\title{Stability of Low-Rank Implicit Regularization in Perturbed Deep Matrix Factorization}

\author[inst1]{Jingzhe Wang}
\ead{jiw148@pitt.edu}

\author[inst2]{Hung-Hsu Chou\corref{cor1}}
\ead{edc93@pitt.edu}

\cortext[cor1]{Corresponding author}

\affiliation[inst1]{organization={Department of Informatics and Networked Systems, University of Pittsburgh}, country={USA}}
\affiliation[inst2]{organization={Department of Mathematics, University of Pittsburgh}, country={USA}}

\begin{abstract}
This paper studies the stability of low-rank implicit regularization in deep matrix factorization, a tractable model for understanding how gradient-based training can favor low-complexity structure. We first revisit the noiseless setting and derive sufficient spectral conditions under which gradient descent exhibits a nonempty low-rank interval. These conditions clarify how the target spectrum, initialization, and step size jointly determine when a low-rank phase is observable along the optimization trajectory. We then analyze the perturbed problem, where the target matrix is subject to an additive perturbation. By studying the perturbed gradient descent dynamics at the eigenvalue level, we prove convergence guarantees and quantify how the perturbation size affects iteration complexity and eigenvalue recovery. Finally, we establish stability of the low-rank phase under perturbation: the effective rank of the iterates remains close to that of the rank-$L$ approximation of the noiseless target over a perturbed low-rank interval, with explicit dependence on the perturbation size. Numerical illustrations support the theoretical predictions and illustrate the role of spectral structure in determining when this stability is observed.
\end{abstract}

\end{frontmatter}

\section{Introduction}
Low-rank structure is a central theme in matrix analysis, signal processing, data approximation, and inverse problems. 
Classical approximation theory shows that truncated singular value decompositions provide optimal low-rank approximations in unitarily invariant norms \cite{eckart1936approximation,mirsky1960symmetric}. 
In recovery problems, low-rank structure is often promoted through convex relaxations, most notably nuclear-norm minimization \cite{candes2009exact,recht2010guaranteed}. 
A complementary line of work studies nonconvex formulations for low-rank matrix problems, including low-rank semidefinite programming and matrix recovery \cite{burer2003nonlinear,tu2016low,bhojanapalli2016global,ge2017no}. 
These works provide algorithmic and geometric foundations for understanding how low-rank structure can be imposed, recovered, or optimized over.

Recently, deep matrix factorization has emerged as another mechanism through which low-rank structure can appear. 
In this approach, one represents the matrix variable as a product of several matrix factors and applies gradient descent to the factors. 
Unlike nuclear-norm regularization or fixed-rank parameterizations, this formulation does not necessarily impose the desired low-rank structure explicitly. 
Instead, empirical and theoretical studies have shown that the optimization dynamics themselves can favor low-rank behavior \cite{gunasekar2017implicit,gunasekar2018implicit,arora2019implicit,woodworth2020kernel,chou2024gradient}. 
This effect is an instance of implicit regularization: structural bias induced by the optimization algorithm in the absence of an explicit regularizer. 
For deep matrix factorization, this bias can be studied through the spectral evolution of the induced matrix iterate. 
In particular, the non-asymptotic analysis of \cite{chou2024gradient} identifies time intervals over which the effective rank of the iterate remains close to that of a prescribed low-rank approximation of the target matrix. 
We refer to such intervals as low-rank intervals.

Building on this analysis, we focus on two questions that are important for understanding low-rank implicit regularization along the gradient descent trajectory. 
First, although the low-rank interval is explicitly characterized in prior work, its dependence on the target spectrum, initialization, and step size is not always transparent from the original formulation. 
To make the low-rank phase more interpretable, we derive sufficient spectral conditions that clarify when such an interval is nonempty. 
Second, existing low-rank interval results are formulated for an exact target matrix. 
In practice, the target matrix is often available only up to perturbation. 
It is therefore natural to ask whether the low-rank phase predicted by the noiseless theory persists when the target matrix is additively perturbed.

Classical matrix perturbation theory controls how spectral quantities of a fixed matrix change under perturbation \cite{stewart1990matrix,bhatia1997matrix,davis1970rotation,wedin1972perturbation}. 
In our setting, one must also track how the perturbation affects the gradient descent trajectory. 
In particular, the perturbation may change the spectral dynamics, the time interval over which low-rank behavior appears, and the quality of the resulting low-rank approximation to the noiseless target. 
Our analysis combines the spectral dynamics of deep matrix factorization with matrix perturbation estimates to quantify these effects.

This leads to the central question of this paper:
\begin{center}
    \textit{Under what spectral conditions does the low-rank phase of gradient descent in deep matrix factorization persist under additive perturbations of the target matrix?}
\end{center}

\subsection{Main Contributions}
This paper analyzes the stability of low-rank implicit regularization in perturbed deep matrix factorization. 
Our contributions are as follows.

\textit{Spectral conditions for low-rank intervals.} We first revisit the noiseless gradient descent dynamics for deep matrix factorization. 
Building on \cite{chou2024gradient}, we derive sufficient spectral conditions under which the noiseless dynamics admit a nonempty low-rank interval. 
These conditions make explicit how the target spectrum, initialization scale, and step size determine whether the low-rank phase is observable.

\textit{Spectral convergence under perturbation.}
We then analyze gradient descent when the target matrix is additively perturbed. 
Under suitable assumptions on the noiseless target and the perturbation, we study the dynamics at the eigenvalue level and prove convergence guarantees for the perturbed problem. 
We also quantify how the perturbation size affects iteration complexity and the recovery of nonnegative eigenvalues.

\textit{Stability of low-rank implicit regularization.}
Using the convergence analysis and matrix perturbation estimates, we prove that the low-rank phase persists under controlled perturbations. 
More precisely, the perturbed dynamics admit a nonempty perturbed low-rank interval over which the effective rank of the iterate remains close to that of the rank-$L$ approximation of the noiseless target, up to an explicit error depending on the perturbation size. 
We further quantify how the perturbation shifts the endpoints of the low-rank interval and bound the resulting low-rank approximation error.

\textit{Numerical illustrations.} We include numerical illustrations to illustrate the spectral conditions for the existence of low-rank intervals and the stability of the low-rank phase under matrix perturbations.

\subsection{Related Work}

\textit{Low-rank matrix approximation and nonconvex factorization. }Classical low-rank approximation and recovery methods provide the broader matrix-analytic context for this work. 
Truncated singular value decompositions characterize optimal low-rank approximations in unitarily invariant norms \cite{eckart1936approximation,mirsky1960symmetric}, while nuclear-norm minimization provides a convex approach to rank-regularized recovery problems \cite{candes2009exact,recht2010guaranteed}. 
Nonconvex factorized formulations, including Burer--Monteiro-type methods and factorized approaches to low-rank recovery, have also been extensively studied \cite{burer2003nonlinear,tu2016low,bhojanapalli2016global,ge2017no}. 
These works typically study low-rank structure imposed by the model formulation, the factor dimension, or the optimization landscape. 
By contrast, our focus is the low-rank behavior that emerges along the gradient descent trajectory in deep matrix factorization, and how this behavior changes when the target matrix is perturbed.

\textit{Implicit regularization in matrix factorization. }
Implicit regularization in matrix factorization and deep linear networks has been studied in several works showing that gradient-based optimization can favor solutions with special spectral or low-complexity structure \cite{gunasekar2017implicit,gunasekar2018implicit,arora2019implicit,woodworth2020kernel}. 
The work most directly related to ours is \cite{chou2024gradient}, which gives a non-asymptotic analysis of gradient descent for deep matrix factorization and identifies low-rank intervals along the optimization trajectory. 
Our work complements this analysis by deriving interpretable sufficient conditions for the existence of such intervals and by proving perturbation-stability estimates for the resulting low-rank phase.

\textit{Matrix perturbation and stability of spectral structure. }
Classical matrix perturbation theory studies how eigenvalues, singular values, and invariant subspaces change under perturbations \cite{stewart1990matrix,bhatia1997matrix}. 
Davis--Kahan-type theorems provide perturbation bounds for eigenspaces \cite{davis1970rotation}, while Wedin's theorem gives analogous estimates for singular subspaces \cite{wedin1972perturbation}. 
Our analysis uses this viewpoint, but the object of study is different from a static perturbation problem. 
We do not merely compare the spectra of the target matrix and its perturbed version; rather, we study how the perturbation changes the gradient descent trajectory, the existence and location of the low-rank interval, and the approximation of the noiseless low-rank target along this trajectory.

\subsection{Organization}

The rest of the paper is organized as follows. 
Section~\ref{sec:prelim} introduces the deep matrix factorization model, the gradient descent dynamics, and the noiseless low-rank implicit regularization result that motivates our analysis. 
Section~\ref{sec:low-rank-interval} derives sufficient spectral conditions for the existence of nonempty low-rank intervals in the noiseless setting. 
Section~\ref{sec:perturbed-convergence} studies spectral convergence of gradient descent under matrix perturbations, including iteration complexity and recovery of nonnegative eigenvalues. 
Section~\ref{sec:stability} proves the stability of low-rank implicit regularization under perturbations. 
Section~\ref{sec:numerics} presents numerical illustrations, and Section~\ref{sec:conclusion} concludes the paper.

\section{Preliminaries}
\label{sec:prelim}

\subsection{Deep Matrix Factorization}

Deep matrix factorization studies matrix approximation through a product of several matrix factors. 
It can be viewed as a matrix-valued analogue of training a deep linear network with squared loss, while avoiding the additional complications introduced by nonlinear activations. 
Given a target matrix $\Wstar$, we consider the optimization problem
\begin{equation}\label{eq:MF}
    \min_{\W_1,\ldots,\W_N}\mathcal{L}(\W_1,\ldots,\W_N),\qquad \mathcal{L}(\W_1,\ldots,\W_N):= \big\| \W_N \cdots \W_1 - \Wstar \big\|_{F}^2.
\end{equation}
The loss in \eqref{eq:MF} is nonconvex in the factors, even though it depends on them only through their product. 
This product structure is precisely what allows the optimization dynamics to induce nontrivial spectral behavior in the matrix iterate
\begin{equation}\label{eq:product-iterate}
    \W(k):=\W_N(k)\cdots \W_1(k).
\end{equation}

In this paper, we focus on gradient descent applied to the factor variables:
\begin{equation}\label{eq:gd}
    \W_j(k+1)=\W_j(k)-\eta\nabla_{\W_j(k)}\mathcal{L}(\W_1(k),\ldots,\W_N(k)),\qquad j=1,\ldots,N.
\end{equation}
Following the setting of \cite{chou2024gradient}, we mainly consider identical positive identity initialization,
\begin{equation}\label{eq:id-init}
    \W_j(0)=\alpha I,\qquad j=1,\ldots,N,
\end{equation}
where $\alpha>0$ is the initialization scale. 
The results of \cite{chou2024gradient} also treat more general non-identical initializations, but the identical initialization case is sufficient for the spectral conditions and perturbation-stability results developed here.

\subsection{Prior Results on Low-Rank Implicit Regularization}

We first recall the noiseless low-rank implicit regularization result from \cite{chou2024gradient}. 
For a matrix $\W$, define its effective rank by
\begin{equation}\label{eq:effective-rank}
    r(\W):=\frac{\|\W\|_*}{\|\W\|},
\end{equation}
where $\|\cdot\|_*$ is the nuclear norm and $\|\cdot\|$ is the spectral norm.

\begin{theorem}[\cite{chou2024gradient}, Theorem~3.5] \label{prop:EffectiveRankDiscretePSD}
    Let $\Wstar \in \mathbb{R}^{n\times n}$ be a symmetric positive semidefinite target matrix with eigenvalues $\lambda_1 \geq \cdots \geq \lambda_n \geq 0$. 
    Assume that $N \geq 2$, $\W_1(k),\dots,\W_N(k) \in \mathbb{R}^{n\times n}$, and $\W(k)=\W_N(k)\cdots \W_1(k)$ follows the gradient descent dynamics \eqref{eq:gd} with initialization $\W_j(0)=\alpha I$ for all $j$.
    Let $L \in [n]$ be fixed and assume $\lambda_{L+1}>0$.
    Let $\varepsilon \in (0,1)$ and $\varepsilon' \in (0,c_N)$, where $c_N=\frac{N-1}{2N-1}$. 
    Assume that $\alpha^N<\varepsilon'\lambda_{L+1}$ and that the stepsize satisfies
    \begin{equation}\label{eq:chou-stepsize}
        \eta < \big[(3N-2)\max\{\alpha^{N-2},\lambda_1^{2-\frac{2}{N}}\}\big]^{-1}.
    \end{equation}
    Define $L'=\max\{\ell\in[n]:\varepsilon'\lambda_\ell>\alpha^N\}$ and $L''=\max\{\ell\in[n]:\lambda_\ell>\alpha^N\}$. Then, for every $k$ satisfying
    \begin{equation}\label{eq:low-rank-window}
        T_0(\{\lambda_\ell\}_{\ell=1}^{L},\varepsilon,\alpha,\eta)
        \leq k
        \leq T_1(\lambda_{L+1},\varepsilon',\alpha,\eta),
    \end{equation}
    we have
    \begin{equation}\label{eq:bound-effective-rank}
		\left|r(\Wstar_L)-r(\W(k))\right|
        \leq
        \varepsilon r(\Wstar_L)
        +\frac{2(L'-L)}{c_N}\frac{\lambda_{L+1}}{\lambda_1}\varepsilon'
        +(n-L')\frac{2\alpha^N}{\varepsilon'\lambda_1},
	\end{equation}
    where $\Wstar_L$ is the best rank-$L$ approximation of $\Wstar$. 
    The quantities $T_0$ and $T_1$ are defined in \cite{chou2024gradient}; in the special case used below, their simplified forms are given in \eqref{eq:T0} and \eqref{eq:T1}.
\end{theorem}
Theorem~\ref{prop:EffectiveRankDiscretePSD} shows that, over the interval specified in \eqref{eq:low-rank-window}, the effective rank of the product iterate remains close to that of the best rank-$L$ approximation of the target matrix. 
We refer to this interval as a low-rank interval. 
For the analysis in the next section, we use the following simplified expressions for the endpoints of this interval in the case $N=2$ and for eigenvalues above the initialization scale:
\begin{equation}
    T_0(\{\lambda_\ell\}_{\ell=1}^{L},\varepsilon,\alpha,\eta)
    =\max\left\{ T_2^{\mathrm{Id}} \left( \lambda_1,\frac{\lambda_1}{2},\alpha,\eta\right), \max_{\ell \in [L]} \; T_2^{\mathrm{Id}} \left(\lambda_\ell, \frac{\sqrt{\lambda_\ell}}{8} \varepsilon ,\alpha,\eta\right) \right\},
    \label{eq:T0}
\end{equation}
\begin{equation}
    T_1(\lambda_{L+1},\varepsilon',\alpha,\eta)
    = \frac{1}{2\eta\lambda_{L+1}}\left[\ln\left(\frac{\lambda_{L+1}}{\alpha^2} -  1\right)-\ln\left(\frac{1}{\varepsilon'}-1\right)\right],
    \label{eq:T1}
\end{equation}
where
\begin{equation}
\begin{split}
    T^{\mathrm{Id}}_2(\lambda,\varepsilon,\alpha,\eta)
    &= \frac{1}{2\eta\lambda}\left[\ln\left(\frac{\lambda}{\alpha^2} -  1\right)-\ln 2\right]
    + \left\lceil \sqrt{\frac{\lambda}{3}}\frac{1}{\alpha} \right\rceil\\
    &\quad+ \frac{\ln(\sqrt{\lambda}/\varepsilon)  - |\ln \left( 1-\sqrt{1/3}\right)|}{|\ln(1- 2\eta\lambda/3)|}.
\end{split}
\label{eq:T2Id}
\end{equation}
The general expressions in \cite{chou2024gradient} are more involved. 
Since the goal of the next section is to identify interpretable spectral conditions for the existence of a low-rank interval, we work with the simplified setting above.

The condition that the low-rank interval is nonempty is
\begin{equation}\label{eq:T0-less-T1}
    T_0(\{\lambda_\ell\}_{\ell=1}^{L},\varepsilon,\alpha,\eta)
    <
    T_1(\lambda_{L+1},\varepsilon',\alpha,\eta).
\end{equation}
Although the endpoints are explicit, the inequality \eqref{eq:T0-less-T1} does not directly reveal how the spectrum, initialization scale, stepsize, and error parameters interact. 
In particular, the lower endpoint $T_0$ depends on the leading eigenvalues $\lambda_1,\ldots,\lambda_L$, while the upper endpoint $T_1$ is governed by $\lambda_{L+1}$. 
Thus, the existence of a nonempty low-rank interval is controlled by the separation between the leading spectrum and the next eigenvalue, together with the optimization parameters. 
The next section derives sufficient spectral conditions that make this dependence explicit.

\begin{remark}
Although the theorem is stated for symmetric positive semidefinite targets, this assumption is mainly used to present the spectral dynamics in a simple form. 
For a nonsymmetric target $\widehat W$, one can consider the self-adjoint dilation,
% $\begin{bsmallmatrix}0&\widehat W\\ \widehat W^\top&0\end{bsmallmatrix}$,
whose eigenvalues are the signed singular values $\pm\sigma_i(\widehat W)$. 
This shows that the spectral viewpoint is not limited to symmetric matrices, although we state and prove our results in the symmetric setting for clarity.
\end{remark}

\section{Low-Rank Intervals}
\label{sec:low-rank-interval}
We will focus on $T_0$ since it involves multiple maximum and hence is more difficult to express explicitly. The expression can be much simplified if $T_2^{\text{Id}}$ exhibits certain monotonicity with respect to $\lambda$.

\begin{lemma}
\label{lemma:T0}
    Suppose $\epsilon\in(0,1)$ and $\lambda_1 > \frac{1}{16}$. If $T_2^{\text{Id}} (\lambda_\ell, \frac{\sqrt{\lambda_\ell}}{8} \epsilon ,\alpha,\eta)$ decreases with respect to $\{\lambda_\ell\}_{\ell\in[L]}$, then $T_0 = T_2^{\text{Id}}\sbrac{\lambda_L,\frac{\sqrt{\lambda_L}}{8}\epsilon, \alpha, \eta}.$
\end{lemma}
\begin{proof}
    Since $\epsilon \in (0,1)$ and $\lambda_1\geq\frac{1}{16}$, $\frac{\lambda_1}{2} > \frac{\sqrt{\lambda_1}}{8}\epsilon$. By \eqref{eq:T2Id}, $T^{\Id}_2(\lambda,\epsilon,\alpha,\eta)$ decreases with respect to $\epsilon$, and hence
    \begin{equation*}
        T_2^{\text{Id}}\sbrac{\lambda_1, \frac{\lambda_1}{2}, \alpha,\eta} 
        \leq T_2^{\text{Id}}\sbrac{\lambda_1, \frac{\sqrt{\lambda_1}}{8}\epsilon, \alpha,\eta}.
    \end{equation*}
    By the decreasing assumption,
    \begin{equation*}
        T_2^{\text{Id}}\sbrac{\lambda_1, \frac{\sqrt{\lambda_1}}{8}\epsilon, \alpha,\eta}
        \leq \max_{\ell\in[L]}T_2^{\text{Id}}\sbrac{\lambda_\ell, \frac{\sqrt{\lambda_\ell}}{8}\epsilon, \alpha,\eta}
        = T_2^{\text{Id}}\sbrac{\lambda_L, \frac{\sqrt{\lambda_L}}{8}\epsilon, \alpha,\eta}.
    \end{equation*}
    This completes the proof.
\end{proof}
The decreasing assumption fits the natural intuition, since recovering small eigenvalues usually require more time. However, it is non-trivial to establish this statement rigorously, and the remainder of this section is devoted to proving that $T_2^{\text{Id}} (\lambda, \frac{\sqrt{\lambda}}{8} \epsilon ,\alpha,\eta)$ decreases with respect to $\{\lambda_\ell\}_{\ell\in[L]}$. For notation simplicity, we introduce a function $T(x)$ defined as 
\begin{equation}
    \label{eq: Tx}
    T(x)=T_2^{\text{Id}}\sbrac{x, \frac{\sqrt{x}}{8}\epsilon, \alpha,\eta}
\end{equation}
where $\epsilon,\epsilon',\alpha,\eta$ are fixed and within the range specified by Theorem \ref{prop:EffectiveRankDiscretePSD} ($\epsilon'$ appears only in $T_1$ so we do not need to worry about it too much for now). Note that the choice of those parameters also imposes restrictions on the range of $x$. In particular, since $\varepsilon' \in (0,\frac{1}{3})$, $\alpha^2 \leq \varepsilon' \lambda_{L+1}$, and $\eta < [4 \max\{1, \lambda_1\}]^{-1}$, we have
\begin{equation}\label{eq:x_domain}
    x\in D=\sbrac{\frac{\alpha^2}{\epsilon'},\frac{1}{4\eta}}
\end{equation}
assuming $\lambda_1\geq 1$, which can always be achieved by re-scaling but is nevertheless stated explicitly for conciseness. Note that the condition $x\in D$ implicitly ensures that $T(x)$ is well-defined. To further simplify the analysis, we decompose the function $T$ into
\begin{equation}
\label{eq:Tx-deco}
    T(x) = A(x) + B(x) + C(x),
\end{equation}
where 
\begin{equation}
\label{eq:Tx-deco-2}
    A(x) = \frac{1}{2\eta x} \sbrac{\ln\sbrac{\frac{x}{\alpha^2}-1}-\ln2},
    \enspace B(x) =  \ceil*{\sqrt{\frac{x}{3}}\frac{1}{\alpha}},
    \enspace C(x) = \frac{K_{\epsilon}}{\abs{\ln \sbrac{1 - \frac{2}{3}\eta x}}}
\end{equation}
and
\begin{equation}
\label{eq:K}
    K_\epsilon:= \ln\frac{8}{\epsilon}-\abs{\ln \sbrac{1 - \sqrt{\frac{1}{3}}}} = \ln \sbrac{\frac{8}{\epsilon}\sbrac{1-\sqrt{\frac{1}{3}}}} > \ln 1 = 0.
\end{equation}
Although $A$ and $C$ are differentiable, $B$ is not continuous and have many jumps. This observation implies that $T$ cannot be decreasing for all $x$, because even if $A$ and $C$ are decreasing functions, occasionally $B$ will cause a drastic increase. However, Lemma \ref{lemma:T0} only requires $T$ to be decreasing with respect to $\{\lambda_\ell\}_{\ell\in[L]}$. We will show that when the gaps between eigenvalues are sufficiently large, the effect of $A+C$ outweighs $B$ and we can still establish the monotonicity statement we are aiming for. To analyze $B$, we define the level set $I_z$ as
\begin{equation}
     I_z
     := \left\{ x \geq0: \ceil*{\sqrt{\frac{x}{3}}\frac{1}{\alpha}} = z \right\}
\end{equation}
for all $z \in \mathbb{N}$.
\begin{lemma}\label{lemma:suff_decay}
    Suppose $x_1 < x_2$ are $m$ levels apart, i.e. $x_1 \in I_{z-m}$ and $x_2 \in I_z$ for some $z \in \mathbb{N}$ and $m\in\mathbb{N}_0$, and $[x_1,x_2]\subset D$ where $D$ is defined in \eqref{eq:x_domain}. Let $\kappa\leq-\max_{x\in[x_1,x_2]}A'(x)+C'(x)$. If 
    \begin{equation}\label{eq:x_gap}
        x_2 - x_1 \geq \frac{m}{\kappa},
    \end{equation}
    then $T(x_1) \geq T(x_2)$.
\end{lemma}
\begin{proof}
    From \eqref{eq:Tx-deco} and \eqref{eq:Tx-deco-2}, $T(x_2) - T(x_1) = A(x_2) + C(x_2) - A(x_1) - C(x_1) + m$. By the Fundamental Theorem of Calculus,
    \begin{align*}
        T(x_2) - T(x_1)
        &= m +\int_{x_1}^{x_2} A'(x)+C'(x)dx\\
        &\leq m + (x_2-x_1)\max_{x\in[x_1,x_2]}\sbrac{A'(x)+C'(x)}\\
        &\leq m - (x_2-x_1)\kappa
        \leq 0.
    \end{align*}
    Hence $T(x_1)\geq T(x_2)$.
\end{proof}
The quantity $\kappa$ in Lemma \ref{lemma:suff_decay} represents the sufficient amount of decrease required from $A+C$ to compensate the increase from $B$. Equation \eqref{eq:x_gap} characterizes the minimal gap required for a net decrease in $T$. Our next step then is to bound $\kappa$ based on the auxiliary lemma \ref{lemma:t_root}

\begin{lemma}\label{lemma:t_root}
    The function $\phi(t)=\ln\sbrac{\frac{t}{2}} - \frac{1}{t} -1$ has a unique positive root, denoted as $t^*$, and $\phi(t)> 0$ for all $t > t^*$. In particular, $t^*<2e+1$.
\end{lemma}
\begin{proof}
    Since
    \begin{equation*}
        \phi'(t)=\frac{1}{t} + \frac{1}{t^2} > 0 ~\text{when}~t>0, \quad 
        \lim_{t \to 0^+}\phi(t) = -\infty, \quad\quad \lim_{t \to \infty^+}\phi(t) = \infty,
    \end{equation*}
    there exists a unique $t^* > 0$ such that 
    \begin{equation*}
        \phi(t^*) = \ln\sbrac{\frac{t^*}{2}}-\frac{1}{t^*} -1 = 0.
    \end{equation*}
    For $t\geq t^*$, because $\phi$ is strictly increasing, $\phi(t)>\phi(t^*)=0$. Due to the increasing nature of $\phi$, $\phi(2e+1)>0$ implies that $t^*<2e+1$.
\end{proof}

\begin{lemma}\label{lemma:A+C_bound}
Suppose $[x_1,x_2]\subset D$ and $x_1\geq 2(e+1)\alpha^2$, where $D$ is defined in \eqref{eq:x_domain}. Then 
\begin{equation}\label{eq:A+C_max}
    \max_{x\in[x_1,x_2]}\;A'(x)+C'(x)\leq -\frac{2\eta K_\epsilon}{3\sbrac{1-\frac{2}{3}\eta x_1}\sbrac{\ln \sbrac{1 - \frac{2}{3}\eta x_2}}^2}.
\end{equation}
\end{lemma}
\begin{proof}
We will analyze $A(x)$ and $C(x)$ separately.
\paragraph{Monotonicity of $A(x)$}
    By the chain rule,
    \begin{equation*}
        A'(x)=\frac{1}{2\eta x^2} \sbrac{\frac{x}{x-\alpha^2}-\ln\sbrac{\frac{x}{\alpha^2}-1}+\ln2}.
    \end{equation*}
    By setting $t = \frac{x}{\alpha^2}-1 > 0$, we obtain the expression
    \begin{equation*}
        A'(x)
        =\frac{1}{2\eta x^2} \sbrac{\frac{t+1}{t}-\ln\sbrac{\frac{t}{2}}}
        =-\frac{\phi(t)}{2\eta x^2} 
    \end{equation*}
    where $\phi$ is defined in Lemma \ref{lemma:t_root}. By assumption, we have
    \begin{equation*}
        t\geq \frac{2(e+1)\alpha^2}{\alpha^2}-1 =2e+1> t^*
    \end{equation*}
    and hence $\phi(t)\geq 0$. Thus $A'(x)\leq 0$.
\paragraph{Monotonicity of $C(x)$}
    By the chain rule,  
    \begin{equation}
    \label{eq:C-diff}
        C'(x) =  -\frac{2\eta K_\epsilon}{3\sbrac{1-\frac{2}{3}\eta x}\sbrac{\ln \sbrac{1 - \frac{2}{3}\eta x}}^2}
    \end{equation}
    where $K_\epsilon>0$ from \eqref{eq:K}. Since $x \in [x_1,x_2]\subset D$, we have
    \begin{equation*}
        1 - \frac{2}{3}\eta x \leq 1 - \frac{2}{3}\eta x_1, \quad\text{and} \quad \left|\ln\sbrac{1-\frac{2}{3}\eta x}\right| \leq \left|\ln\sbrac{1-\frac{2}{3}\eta x_2}\right|
    \end{equation*}
    This implies that for all $x\in[x_1,x_2]$, 
    \begin{equation}
    \label{eq:lbd-C'}
        C'(x) \leq
        -\frac{2\eta K_\epsilon}{3\sbrac{1-\frac{2}{3}\eta x_1}\sbrac{\ln \sbrac{1 - \frac{2}{3}\eta x_2}}^2}.
    \end{equation}
    Combining this with monotonicity of $A$, we arrive at our conclusion.
\end{proof}
Note that the upper bound in \eqref{eq:A+C_max} only depends on the end points on the interval $x_1$ and $x_2$ and not the exact levels they lie. Hence the the decaying requirement $\kappa$ can be written as a function of $(x_1,x_2)$. Combining Lemma \ref{lemma:T0}, \ref{lemma:suff_decay}, and \ref{lemma:A+C_bound}, we obtain an condition for an explicit form of $T_0$.
\begin{lemma}\label{lemma:suff_decay_gap}
    Consider the same setting as in Theorem \ref{prop:EffectiveRankDiscretePSD} with $N=2$. Suppose $\lambda_1\geq 1$ and $\lambda_{L}\geq 2(e+1)\alpha^2$. Set
    \begin{equation}
        \kappa(\lambda_{\ell},\lambda_{\ell+1}) = \frac{2\eta K_\epsilon}{3\sbrac{1-\frac{2}{3}\eta \lambda_{\ell+1}}\sbrac{\ln \sbrac{1 - \frac{2}{3}\eta \lambda_{\ell}}}^2}.
    \end{equation}
    If $\{\lambda_\ell\}_{\ell=1}^L$ satisfies
    \begin{equation}\label{eq:x_gap2}
        \lambda_{\ell} - \lambda_{\ell+1} \geq \frac{1}{\kappa(\lambda_\ell,\lambda_{\ell+1})}\cdot \sbrac{\ceil*{\sqrt{\frac{\lambda_{\ell}}{3}}\frac{1}{\alpha}}-\ceil*{\sqrt{\frac{\lambda_{\ell+1}}{3}}\frac{1}{\alpha}}},
    \end{equation}
    then $T_0 = T_2^{\text{Id}}\sbrac{\lambda_L,\frac{\sqrt{\lambda_L}}{8}\epsilon, \alpha, \eta}$.
\end{lemma}
\begin{proof}
    By Lemma \ref{lemma:A+C_bound}, $\max_{x\in[\lambda_{\ell+1},\lambda_{\ell}]}\;A'(x)+C'(x)\leq \kappa(\lambda_\ell,\lambda_{\ell+1})$. By the definition of level sets $I_z$, $m=\ceil*{\sqrt{\frac{\lambda_{\ell}}{3}}\frac{1}{\alpha}}-\ceil*{\sqrt{\frac{\lambda_{\ell+1}}{3}}\frac{1}{\alpha}}$. Since $\lambda_{\ell}-\lambda_{\ell+1}\geq \frac{m}{\kappa}$, we have $T(\lambda_\ell) \geq T(\lambda_{\ell+1})$ by Lemma \ref{lemma:suff_decay} for all $\ell\in[L-1]$. We complete the proof by applying Lemma \ref{lemma:T0}, which yields $T_0 = T_2^{\text{Id}}\sbrac{\lambda_L,\frac{\sqrt{\lambda_L}}{8}\epsilon, \alpha, \eta}$.
\end{proof}

Now we have the explicit formula for $T_0 = T_2^{\text{Id}}\sbrac{\lambda_L,\frac{\sqrt{\lambda_L}}{8}\epsilon, \alpha, \eta}$, we can compare it with $T_1$. Recall from \eqref{eq:T1} and \eqref{eq:T2Id} that
\begin{align*}
    T_0
    &= \frac{1}{2\eta\lambda_L}\left[\ln\left(\frac{\lambda_L}{\alpha^2} -  1\right)-\ln 2\right]
    + \left\lceil \sqrt{\frac{\lambda_L}{3}}\frac{1}{\alpha} \right\rceil 
    + \frac{K_\epsilon}{|\ln(1- 2\eta\lambda_L/3)|},\\
    T_1
    &= \frac{1}{2\eta\lambda_{\L+1}}\left[\ln\left(\frac{\lambda_{\L+1}}{\alpha^2} -  1\right)-\ln\left(\frac{1}{\epsilon'}-1\right)\right].
\end{align*}
Note that as $\eta$ goes to zero, $T_1$ is dominated by the factor $\frac{1}{\eta\lambda_{L+1}}$ while $T_1$ is dominated by the factor $\frac{1}{\eta\lambda_{L+1}}$. Hence as long as $\lambda_L-\lambda_{L+1}>0$, there exists sufficiently small $\eta$ such that $T_1-T_0>0$. This fits into the intuition that implicit regularization is more likely to appear with smaller step sizes. 

\begin{theorem}
\label{thm:spectral-time}
    Consider the same setting as in Theorem \ref{prop:EffectiveRankDiscretePSD} with $N=2$. Define $\L' = \max\{ \l \in [n] \colon \varepsilon' \lambda_\l > \alpha^N \}$. Then, for $k\in[T_0,T_1]$, the gradient descent satisfies the low-rank approximation
    \begin{align}
    \label{eq:eff-rank-N-2}
		\left|r(\Wstar_L) - r(W(k))\right|
        \leq
        \varepsilon r(\Wstar_\L) + 6(L'-L)
        \frac{\lambda_{L+1}}{\lambda_1} \varepsilon'
        +(n-L')\frac{2\alpha^2}{ \varepsilon'\lambda_1}.
	\end{align}
    where $\Wstar_\L$ is the best rank-$\L$ approximation of $\Wstar$ and $r(\W) = \frac{\|\W\|_*}{\|\W\|}$ is the effective rank. Moreover, $T_0$ and $T_1$ can be specified more precisely and shown to satisfy $T_0<T_1$ under the following conditions. 
    \begin{enumerate}
        \item Suppose $\lambda_1\geq 1$ and $\lambda_{L}\geq 2(e+1)\alpha^2$. Set
        \begin{equation}
            \kappa(\lambda_{\ell+1},\lambda_{\ell})
            = \frac{2\eta K_\epsilon}{3\sbrac{1-\frac{2}{3}\eta \lambda_{\ell+1}}\sbrac{\ln \sbrac{1 - \frac{2}{3}\eta \lambda_{\ell}}}^2}
        \end{equation}
        with $K_\epsilon = \ln(8/\epsilon)-\abs{\ln \sbrac{1 - \sqrt{1/3}}}$. If $\{\lambda_\ell\}_{\ell=1}^L$ satisfies
        \begin{equation}\label{eq:x_gap3}
            \lambda_{\ell} - \lambda_{\ell+1} \geq \frac{1}{\kappa(\lambda_\ell,\lambda_{\ell+1})}\cdot \sbrac{\ceil*{\sqrt{\frac{\lambda_{\ell}}{3}}\frac{1}{\alpha}}-\ceil*{\sqrt{\frac{\lambda_{\ell+1}}{3}}\frac{1}{\alpha}}},
        \end{equation}
        then the time $T_0$ and $T_1$ have the explicit expression
        \begin{align}
            T_0
            &= \frac{1}{2\eta\lambda_L}\left[\ln\left(\frac{\lambda_L}{\alpha^2} -  1\right)-\ln 2\right]
            + \left\lceil \sqrt{\frac{\lambda_L}{3}}\frac{1}{\alpha} \right\rceil 
            + \frac{K_\epsilon}{|\ln(1- 2\eta\lambda_L/3)|},\\
            T_1
            &= \frac{1}{2\eta\lambda_{\L+1}}\left[\ln\left(\frac{\lambda_{\L+1}}{\alpha^2} -  1\right)-\ln\left(\frac{1}{\epsilon'}-1\right)\right].
        \end{align}
        \item Moreover, if $\alpha<\alpha^*$ with
        \begin{equation}
        \label{eq:alpha-bound}
            \alpha^*:=\text{exp}\left(\frac{1}{2(\lambda_{L}-\lambda_{L+1})}
            (\lambda_{\L}\ln\sbrac{\epsilon'\lambda_{L+1}} 
            - \lambda_{L+1}(\ln\lambda_L -\ln 2-3K_\epsilon))\right)
        \end{equation}
        and $\eta < \eta^*$ with
        \begin{equation}
        \label{eq:eta-bound}
            \eta^*
            :=\frac{\frac{1}{\lambda_{\L+1}}\left[\ln\left(\frac{\lambda_{\L+1}}{\alpha^2} -  1\right)-\ln\left(\frac{1}{\epsilon'}-1\right)\right] - \frac{1}{\lambda_L}\left[\ln\left(\frac{\lambda_L}{\alpha^2} -  1\right)-\ln 2-3K_\epsilon\right]}{2\left\lceil \sqrt{\frac{\lambda_L}{3}}\frac{1}{\alpha} \right\rceil} ,
        \end{equation}
        then $T_0<T_1$, i.e. the size of the window for observation implicit regularization is strictly positive.
    \end{enumerate}
\end{theorem}
\begin{proof}
    The first statement follows from Lemma \ref{lemma:suff_decay_gap}. Here we aim to prove the second statement. The explicit formula of $T_0$ and $T_1$ are given by Lemma \ref{lemma:suff_decay_gap}, \eqref{eq:T1} and \eqref{eq:T2Id}. The gap $T_1-T_0$ can be written as
    \begin{equation*}
        \frac{\ln\big(\frac{\lambda_{\L+1}}{\alpha^2} -  1\big)-C_{\epsilon'}}{2\eta\lambda_{\L+1}}
        - \frac{\ln\big(\frac{\lambda_L}{\alpha^2} -  1\big)-\ln 2}{2\eta\lambda_L}
        - \left\lceil \sqrt{\frac{\lambda_L}{3}}\frac{1}{\alpha} \right\rceil 
        - \frac{K_\epsilon}{|\ln(1- 2\eta\lambda/3)|}
    \end{equation*}
    where $C_{\epsilon'}=\ln((\epsilon')^{-1}-1)$. A cleaner expression can be attained by the inequality $-\frac{1}{|\ln(1-x)|}\geq-\frac{1}{x}$ for $x\in(0,1)$, i.e.
    \begin{equation*}
    T_1-T_0\geq
    \frac{1}{2\eta}\underbrace{\left(\frac{\ln\big(\frac{\lambda_{\L+1}}{\alpha^2} -  1\big)-C_{\epsilon'}}{\lambda_{\L+1}}
    - \frac{\ln\big(\frac{\lambda_L}{\alpha^2} -  1\big)-\ln 2-3K_\epsilon}{\lambda_L}\right)}_{=:g_1}
    - \underbrace{\left\lceil \sqrt{\frac{\lambda_L}{3}}\frac{1}{\alpha} \right\rceil}_{=:g_0}.
    \end{equation*}
    Since $g_0>0$, for the gap to be positive we need $g_1>2\eta g_0$, which is possible if and only if $g_1>0$ as well. We now show that $g_1>0$ by using other assumptions on the parameters. In particular, we will focus on $\alpha$ since we have the freedom to choose the initialization. Because $\alpha^2 \leq \epsilon' \lambda_{L+1}$, $\lambda_{L+1} - \alpha^2 \geq \sbrac{1-\epsilon'} \lambda_{L+1}$, and hence
    \begin{equation*}
        \ln\left(\frac{\lambda_{\L+1}}{\alpha^2} -  1\right)
        =\ln(\lambda_{\L+1}-\alpha^2)-\ln(\alpha^2)
        \geq \ln\sbrac{\sbrac{1-\epsilon'}\lambda_{L+1}} -2\ln\alpha.
    \end{equation*}
    By monotonicity of the logarithm,
    \begin{equation*}
        \ln\left(\frac{\lambda_L}{\alpha^2} -  1\right)
        \leq
        \ln\left(\frac{\lambda_L}{\alpha^2}\right)
        =\ln\lambda_L - 2\ln\alpha.
    \end{equation*}
    Combining the above, we obtain
    \begin{equation}
        g_1
        \geq \frac{2(\lambda_{L+1}-\lambda_L)\ln\alpha}{\lambda_L\lambda_{\L+1}}
        +\frac{1}{\lambda_{\L+1}}\ln\sbrac{\epsilon'\lambda_{L+1}}- \frac{1}{\lambda_L}(\ln\lambda_L -\ln 2-3K_\epsilon)
    \end{equation}
    which is strictly positive according to our assumption. Thus for $\eta<\frac{g_1}{2g_0}$, $T_1-T_0>0$ so there exists a window for which implicit regularization of gradient descent is observable.
\end{proof}

\section{Spectral Convergence under Matrix Perturbations}
\label{sec:perturbed-convergence}
\subsection{Noise Model}
Recall the deep matrix factorization optimization problem in \eqref{eq:MF} where $\Wstar$ is the ground-truth symmetric matrix. Starting from section, we assume that $\Wstar$ is not available instead we only has its perturbed version $\wt{W}$, i.e $\wt{W} = \Wstar + E$. Here $E \in \mathbb{R}^{n\times n}$ is a symmetric noisy matrix constructed by
\begin{equation*}
     E_{ij}\overset{\text{i.i.d}}{\sim} \mathcal{N}(0,\sigma^2)\quad\quad \text{for}~1\leq i \leq j \leq n, \quad\quad E_{ji} = E_{ij}~\text{for}~i < j
\end{equation*}
Introducing the noisy matrix $\wt{W}$ gives rise to the noisy matrix factorization problem
\begin{equation}
    \label{eq:dp-est}
    \min_{W_1,...,W_N} \wt{\mathcal{L}}\sbrac{W_1, \cdots, W_N}, \quad  \wt{\mathcal{L}}\sbrac{W_1, \cdots, W_N} =\| W_N...W_1- \widetilde{W}\|^2_F, \quad\quad  
\end{equation}
% \subsection{Algorithm: \ngd}
By following the identical initialization $W_j(0) = \alpha I$ for all $j$, the gradient descent dynamics for solving \eqref{eq:dp-est} is
\begin{equation}
\label{eq:dp-update}
    W_j(k+1) = W_j(k) - \eta \nabla_{W_j(k)}\widetilde{\mathcal{L}}(W(k)) 
    % W_j(0) &= \alpha W_0 \label{eq:dp-init}\\
    % \widetilde{\mathcal{L}}(W) &= \frac{1}{2}\|W - \widetilde{W}\|_F^2 \quad \nabla_W\widetilde{\mathcal{L}}(W) = (W - \widetilde{W}) \label{eq:dp-loss}
\end{equation}
Our goal is then to give quantitative results on how does the perturbation impact the recovery of a low-rank approximation of the original $\widehat{W}$ and the observability results of implicit regularization presented in Section~\ref{sec:low-rank-interval}.
% \subsection{Dynamics of \ngd }
\begin{lemma}\label{lemma:dp-dynammics}
~Let $(W_j(k))_{j=1}^N$ be the solution to the gradient descent~(\ref{eq:dp-update}) for Problem~($\ref{eq:dp-est}$) with identical initialization $W_j(0) = \alpha I$ for all $j$. Let $\widetilde{W} = \widetilde{V}\widetilde{\Sigma} \widetilde{V}^{\top}$ be an eigenvalue decomposition of the noisy symmetric ground-truth matrix $\widetilde{W}$, where $\widetilde{V}$ is orthogonal.Then the matrices $D_j(k) := \widetilde{V}^{\top} W_j(k)\widetilde{V}$ are real, diagonal, and identical, i.e., $D_j(k) = D(k)$ for all $j \in [N]$ for some $D(k)$, and follow the dynamics:
\begin{equation}
    D(k+1) = D(k) -\eta D(k)^{N-1}(D(k)^N - \widetilde{\Sigma}),\quad k \in \mathbb{N}_0.
\end{equation}
\end{lemma}
Lemma~\ref{lemma:dp-dynammics} follows by the same induction argument as in \cite[Lemma 2.1]{chou2024gradient}, with $V$ replaced by $\widetilde{V}$; we therefore omit the proof. As a consequence of Lemma~\ref{lemma:dp-dynammics}, each diagonal entry $d(k)$ of $D(k)$ satisfies
\begin{equation}
\label{eq:dp-diag-evol}
    d(k+1) = d(k) - \eta d(k)^{N-1}\sbrac{d(k)^N - \widetilde{\lambda}},
\end{equation}
% where $d(0) = \alpha >0$ and $\widetilde{\lambda}$ is an eigenvalue of the noisy ground truth matrix $\widetilde{W}$. Before detailing convergence of $d(k)$ to a true eigenvalue $\lambda$, we first derive explicit bound for $|\widetilde{\lambda} - \lambda|$.
Since $\widehat{W}$ and $E$ are both symmetric in $\mathbb{R}^{n\times n}$, by the Weyl's inequality \cite{horn2012matrix}, eigenvalues are upper-bounded as
\begin{equation}
\label{eq:lambda-bound}
     |\widetilde{\lambda} - \lambda| \leq \|E\|
\end{equation}
In addition, by the concentration inequality for random sub-gaussian matrices from \cite[Corollary 4.4.8]{vershynin2018high}, 
with probability at least $1-\delta'$,
\begin{equation}
\label{eq:E-bound}
     \|E\| \leq \ebound, \quad \ebound := C \sigma
     \left(\sqrt{n}+ \sqrt{\ln{\frac{4}{\delta'}}}\right)
\end{equation}
where $C$ is some absolute constant.
\subsection{Convergence Analysis}
In this section, given \eqref{eq:dp-diag-evol}, we aim to first give new results that captures $d(k)$'s convergence towards $\lambda$. Our proof strategy builds atop the convergence result of \eqref{eq:dp-diag-evol} governed by \cite[Lemma 2.2]{chou2024gradient}. By combining it with \eqref{eq:lambda-bound}, we can explicitly quantify the neighborhood of $\lambda$ to which $d(k)$ finally converges. Then we analyze the rate of convergence via $T^{\Id}_N(\nl,\widetilde{\epsilon},\alpha,\eta)$~(full expression see \cite[Equation~(22)]{chou2024gradient}), which quantifies the number of iterations required by \eqref{eq:dp-update} to achieve the accuracy level $\widetilde{\epsilon}$.
\begin{lemma}
\label{lemma:convergence}
Let $\lambda \in \mathbb{R}$ be an eigenvalue of $\widehat{W}$ and $\widetilde{\lambda} \in \mathbb{R}$ be an eigenvalue of $\widetilde{W}$. Denote $\widetilde{\lambda}_{+} = \max\{\widetilde{\lambda},0\}$ and $\lambda_{+} = \max\{\lambda,0\}$. Let $d$ be the solution of (\ref{eq:dp-diag-evol}) and $\alpha > 0$. We distinguish two cases, $N=1$ and $N\geq 2$:
\begin{enumerate}
    \item If $N = 1$ and $\eta \in (0,1)$,
    $|d(k) - \lambda| \leq (1-\eta)^k(\alpha - \widetilde{\lambda})+\lnorm{E}, \quad \text{for all}~k \in \mathbb{N}.$
    \item If $N \geq 2$ and 
     \begin{align*}
        0 < \eta < 
        \begin{cases}
         %\frac{1}{
         \left( N \max \left\{ \alpha,|\widetilde{\lambda}|^{\frac{1}{N}} \right\}^{2N-2} \right)^{-1} & \mbox{ if } \widetilde{\lambda} > 0,\\
        \alpha^{-2N+2} & \mbox{ if } \widetilde{\lambda} \leq 0 \mbox{ \rm{and} } \alpha \geq |\widetilde{\lambda}|^{\frac{1}{N}}, \\
         \left( (3N-2) |\widetilde{\lambda}|^{2-\frac{2}{N}}\right)^{-1} & \mbox{ if } \widetilde{\lambda} < 0 \mbox{ \rm{and} } 0 < \alpha < |\widetilde{\lambda}|^{\frac{1}{N}}, \\
        \end{cases}
    \end{align*}
    the limit $\lim_{k \to \infty}d(k)$ satisfies
    \[ \left|\lim_{k \to \infty}d(k) - \lambda_{+}^{\frac{1}{N}}\right| \leq \frac{1}{N}(\min\{ \widetilde{\lambda}_{+}, \lambda_{+}\})^{\frac{1}{N}}\lnorm{E}.\]
    Moreover, we have that for all $k \in \mathbb{N}$, $d(k) \in [\alpha, \left(\lambda + \lnorm{E}\right)^{\frac{1}{N}}]$ if $\nl \geq \alpha^{N} + \lnorm{E}$, and $d(k) \in [ \lambda_{+}^{\frac{1}{N}} - \lnorm{E}, \alpha ]$ if $\nl < \alpha^{N}$.
\end{enumerate}
\end{lemma}
\begin{proof}
For case $N=1$, by the triangle inequality, $\abs{d(k) - \lambda}$ satisfies
    \begin{equation*}
        |d(k) - \lambda| \leq |d(k) - \widetilde{\lambda}| +  |\widetilde{\lambda} - \lambda|
                         \leq (1-\eta)^k(\alpha-\widetilde{\lambda}) + |\widetilde{\lambda} - \lambda|
                         \leq (1-\eta)^k(\alpha-\widetilde{\lambda}) +\lnorm{E}
    \end{equation*}
where the second inequality holds due to \cite[Lemma 2.2]{chou2024gradient}, and the third inequality satisfies from $\lvert\lambda-\nl\rvert$ is bounded as in \eqref{eq:lambda-bound}. For case $N\geq 2$, by the triangle inequality, 
\begin{equation*}
        \left|\lim_{k \to \infty}d(k) - \lambda_{+}^{\frac{1}{N}} \right| \leq  \left|\lim_{k \to \infty}d(k) - \widetilde{\lambda}_{+}^{\frac{1}{N}}\right| + \left|\widetilde{\lambda}_{+}^{\frac{1}{N}} - \lambda_{+}^{\frac{1}{N}}\right|
        \leq\left|\widetilde{\lambda}_{+}^{\frac{1}{N}} - \lambda_{+}^{\frac{1}{N}}\right|
\end{equation*}
where the second inequality follows from \cite[Lemma 2.2]{chou2024gradient}, saying that when $N\geq2$, $\vert\lim_{k \to \infty}d(k) - \widetilde{\lambda}_{+}^{\frac{1}{N}}\rvert =0$. It remains to upper bound $\lvert\widetilde{\lambda}_{+}^{\frac{1}{N}} - \lambda_{+}^{\frac{1}{N}}\rvert$. Let $a:= \min\{\widetilde{\lambda}_{+}, \lambda_{+}\}$ and $b:= \max\{\widetilde{\lambda}_{+}, \lambda_{+}\}$. Consider the smooth function $f(x) = x^{\frac{1}{N}}$ defined on $[a,b] \subset [0, \infty)$.
By the Mean Value Theorem, there exist $\xi \in [a,b]$ such that $ b^{\frac{1}{N}} - a^{\frac{1}{N}} = \frac{1}{N}\xi^{\frac{1}{N}-1}(b-a).$
Since $a \leq \xi \leq b$, $b^{\frac{1}{N}} - a^{\frac{1}{N}} \leq \frac{1}{N}a^{\frac{1}{N}-1}(b-a).$
Taking absolute value on both sides,$\lvert b^{\frac{1}{N}} - a^{\frac{1}{N}}\rvert \leq \frac{1}{N}a^{\frac{1}{N}-1}|b-a|$.
By $b^{\frac{1}{N}}-a^{\frac{1}{N}} = \lvert\widetilde{\lambda}_{+}^{\frac{1}{N}}  - \lambda_{+}^{\frac{1}{N}}\rvert$, $b - a = \lvert\widetilde{\lambda}_{+}  - \lambda_{+}\rvert $, and $\lvert\widetilde{\lambda}_{+}  - \lambda_{+}\rvert \leq \lvert\widetilde{\lambda}  - \lambda\rvert$,  
    \begin{equation*}
        \lvert \widetilde{\lambda}_{+}^{\frac{1}{N}}  - \lambda_{+}^{\frac{1}{N}}\rvert \leq \frac{1}{N}a^{\frac{1}{N}}\lvert\widetilde{\lambda}_{+}  - \lambda_{+}\rvert \leq \frac{1}{N}a^{\frac{1}{N}}\lvert\widetilde{\lambda}  - \lambda\rvert \leq \frac{1}{N}a^{\frac{1}{N}}\lnorm{E}.
    \end{equation*}
    Thus
    \begin{equation}
    \label{eq:eigen-diff-sq}
        \left|\lim_{k \to \infty}d(k) - \lambda_{+}^{\frac{1}{N}} \right|
        \leq\left|\widetilde{\lambda}_{+}^{\frac{1}{N}} - \lambda_{+}^{\frac{1}{N}}\right| \leq \frac{1}{N}a^{\frac{1}{N}}\lnorm{E}.
    \end{equation} 
    Finally, we derive the range of $d(k)$ for all $k \in \mathbb{N}$. From \cite[Lemma 2.2]{chou2024gradient}, we have that if $\nl \geq \alpha^N$, $d(k) \in [\alpha, \nl^{\frac{1}{N}}]$.  Since we know that $ \lvert\nl - \lambda\rvert\leq \lnorm{E}$, we have $\nl \in [ \lambda - \lnorm{E}, \lambda + \lnorm{E}]$. Thus, $d(k) \in  [\alpha, (\lambda+\lnorm{E})^{\frac{1}{N}}]$. Also, if $\nl < \alpha^{N}$, we have $d(k) \in [\nl_{+}^{\frac{1}{N}},\alpha]$. From $\lvert\nl_{+}^{\frac{1}{N}} - \lambda_{+}^{\frac{1}{N}}\rvert \leq \lnorm{E}$, we have $d(k) \in [(\lambda_{+} - \lnorm{E})^{\frac{1}{N}}, \alpha]$. This completes the proof.
\end{proof}

\begin{theorem}
\label{thm:iter-complexity}
Let $N \geq 2$, $\nl, \lambda \in \mathbb{R}$, $\alpha>0$, $\eta > 0$. Let $d$ be the solution of (\ref{eq:dp-diag-evol}), and define $M = \max\{\alpha, |\nl|^{\frac{1}{N}}\}$. Suppose the step size $\eta$ satisfies
\begin{align} \label{eq:cond:stepsize}
     0 < \eta <
     \begin{cases}
     \displaystyle\frac{1}{2NM^{2N-2}}, & \text{if } \nl \geq 0,\\[1ex]
     \displaystyle\frac{1}{(3N-2)M^{2N-2}}, & \text{if } \nl < 0.
     \end{cases}
\end{align}
Define $\beta = \lvert\nlrpr - \lpnr\rvert$. Let $\widetilde{\epsilon} \in (0, |\alpha - \nlrp^{\frac{1}{N}}|)$ and $\epsilon' = \widetilde{\epsilon}+\beta$.
Let 
\[ \widetilde{T} := \min\{k : |d(k) - \nlrp^{\frac{1}{N}}|\leq \wt{\epsilon}\}, \quad\quad T := \min\{k : |d(k) - \lambda_+^{\frac{1}{N}}|\leq \epsilon' \} \] be the minimal number of iterations required to reach accuracy $\wt{\epsilon}$ for $\nlrp^{\frac{1}{N}}$ and $\epsilon'$ for $\lambda_{+}^{\frac{1}{N}}$, respectively. Then 
\begin{equation*}
    T \leq \widetilde{T} \leq  T^{\Id}_N(\nl,\wt{\epsilon},\alpha,\eta).
\end{equation*}
Moreover, if $\nl > \alpha^N$, then $T$ satisfies the lower bound 
\[
T \geq 
% \widetilde{T}(\wt{\epsilon}+2\beta) \;\ge\;
\begin{cases}
\displaystyle
\frac{1}{\eta}\,T_N^+\!\bigl(\nl,\,(c_N\nl)^{\frac{1}{N}},\,\alpha\bigr)
\;+\;
\frac{\ln\!\bigl(\nl^{\frac{1}{N}}/(\wt{\epsilon}+2\beta)\bigr)\;-\;b_N}
     {\displaystyle\bigl|\ln\!\bigl(1-\eta\,N\,\nl^{2-\frac{2}{N}}\bigr)\bigr|},
& \text{if } \alpha^N < c_N\,\nl,\\[2ex]
\displaystyle
\frac{\ln\!\bigl((\nl^{\frac{1}{N}}-\alpha)/(\wt{\epsilon}+2\beta)\bigr)}
     {\displaystyle\bigl|\ln\!\bigl(1-\eta\,N\,(c_N\nl)^{2-\frac{2}{N}}\bigr)\bigr|},
& \text{if } \alpha^N \ge c_N\,\nl.
\end{cases}
\]
where $ c_N := \frac{N-1}{2N-1}$ and $b_N := \left| \ln\left(\frac{1}{2c_N} - c_N^{\frac{1}{N}}\right)\right|,$ and $T_N^+\!\bigl(\nl,\,(c_N\nl)^{\frac{1}{N}},\,\alpha\bigr)$ is defined in \cite[Equation (21)]{chou2024gradient}.
\end{theorem}
\begin{proof}
We define $\widetilde{T} = \min\{k : |d(k) - \nlrp^{\frac{1}{N}}| \leq \widetilde{\epsilon}\}$ as the minimal number of iterations to achieve $\widetilde{\epsilon}$ error. We first prove upper bound for $T$.\\
\textbf{Upper Bound of $T$:}~By applying the convergence rate result with respect to $\nlrpr$ from \cite[Theorem 2.4]{chou2024gradient}, we know that $\widetilde{T} \leq T^{\Id}_N(\nl,\widetilde{\epsilon},\alpha,\eta)$ and $\lvert d(k) - \nlrpr\rvert \leq \widetilde{\epsilon}$. Using the triangle inequality, 
\begin{equation*}
    |d(k) - \lambda_{+}^{\frac{1}{N}}| \leq \lvert d(k) - \nlrpr\rvert + \lvert\nlrpr - \lpnr\rvert \leq \widetilde{\epsilon} +  \beta.
\end{equation*}
Thus, whenever $\lvert d(k) - \nlrpr\rvert \leq \wt{\epsilon}$, we also have $ |d(k) - \lambda_{+}^{\frac{1}{N}}| \leq \epsilon'$ with $\epsilon' = \wt{\epsilon}+\beta \in (0, |\alpha - \lambda_{+}^{\frac{1}{N}}| + 2\beta) $. This means for any index $k \in \mathbb{N}_0$ satisfying $\lvert d(k) - \nlrpr\rvert \leq \wt{\epsilon}$, the index also satisfies $ |d(k) - \lambda_{+}^{\frac{1}{N}}| \leq \epsilon'$. So we have 
\begin{equation*}
    \left\{ k: \lvert d(k) - \nlrpr\rvert \leq \wt{\epsilon} \right\} \subseteq  \left\{ k: \lvert d(k) - \lpnr\rvert \leq \wt{\epsilon} + \beta\right\}.
\end{equation*}
Taking minima on both sides gives
\begin{equation*}
     \widetilde{T} =\min\left\{ k: \lvert d(k) - \nlrpr\rvert \leq \wt{\epsilon} \right\} \geq   \min\left\{ k: \lvert d(k) - \lpnr\rvert \leq \wt{\epsilon} + \beta\right\} = T.
\end{equation*}
Since $\widetilde{T} \leq T^{\Id}_N(\nl,\wt{\epsilon},\alpha,\eta) $, finally we have $ T \leq \widetilde{T} \leq  T^{\Id}_N(\nl,\wt{\epsilon},\alpha,\eta)$.\\
\textbf{Lower Bound of $T$:}
% \noteE{Minor comment: the following argument can be written more compactly by setting $\epsilon'=\epsilon''+\beta$ and just apply what we have above with symmetry.}
Next, we derive the lower bound of $T$. In this case, we need $\lvert d(k) - \lpnr\rvert \leq \wt{\epsilon} + \beta $. Using the triangle inequality for $\lvert d(k) - \nlrpr\rvert  $,
\begin{equation*}
    \lvert d(k) - \nlrpr\rvert   \leq \lvert d(k) - \lpnr\rvert  + \lvert  \lpnr - \nlrpr\rvert  \leq \wt{\epsilon} + 2\beta.
\end{equation*}
This means that for $k$ such that $\lvert d(k) - \lpnr\rvert  \leq \wt{\epsilon}+\beta$, $ \lvert d(k) - \nlrpr\rvert   \leq \wt{\epsilon} + 2\beta$ holds. Thus, we have 
\begin{equation*}
     \left\{ k: \lvert  d(k) - \lpnr\rvert  \leq \wt{\epsilon}+\beta\right\} \subseteq \left\{ k: \lvert d(k) - \nlrpr\rvert \leq \wt{\epsilon} + 2\beta \right\}.
\end{equation*}
Taking $\min$ on both sides,
\begin{equation*}
    T = \min\left\{ k: \lvert  d(k) - \lpnr\rvert  \leq \wt{\epsilon} + \beta\right\} \geq \min\left\{ k: \lvert  d(k) - \nlrpr \rvert  \leq \wt{\epsilon}+ 2\beta \right\} = \widetilde{T}(\wt{\epsilon}+2\beta)
\end{equation*}
where $\widetilde{T}(\wt{\epsilon}+2\beta)$ denotes the minimal iterations needed for $d(k)$ to achieve  error $\wt{\epsilon} + 2\beta$ bound when heading towards $\nlrpr$. By \cite[Theorem 2.4, Eq~(25)]{chou2024gradient}, when $\nl > \alpha^N$,
\[
\widetilde{T}(\wt{\epsilon}+2\beta) \;\ge\;
\begin{cases}
\displaystyle
\frac{1}{\eta}\,T_N^+\!\bigl(\nl,\,(c_N\nl)^{\frac{1}{N}},\,\alpha\bigr)
\;+\;
\frac{\ln\!\bigl(\nl^{\frac{1}{N}}/(\wt{\epsilon}+2\beta)\bigr)\;-\;b_N}
     {\displaystyle\bigl|\ln\!\bigl(1-\eta\,N\,\nl^{2-\frac{2}{N}}\bigr)\bigr|},
& \text{if } \alpha^N < c_N\,\nl,\\[2ex]
\displaystyle
\frac{\ln\!\bigl((\nl^{\frac{1}{N}}-\alpha)/(\wt{\epsilon}+2\beta)\bigr)}
     {\displaystyle\bigl|\ln\!\bigl(1-\eta\,N\,(c_N\nl)^{2-\frac{2}{N}}\bigr)\bigr|},
& \text{if } \alpha^N \ge c_N\,\nl.
\end{cases}
\]
This completes the proof.
\end{proof}

\subsection{Recovery of Non-negative Eigenvalues}
We are now ready to demonstrate that, the gradient descent in \eqref{eq:dp-update} can still recover non-negative eigenvalues of the noiseless ground-truth matrix $\hW$. Next results will give quantitative guarantee on how the introduced noise impact the recovery.
\begin{theorem}
\label{thm:dp-pos-eigen-recov}
    Let $N \geq 2$, $\alpha > 0$, $\widetilde{\epsilon} > 0$. Let $\hW \in \mathbb{R}^{n \times n}$ be an unknown symmetric ground-truth matrix with with eigenvalues $\lambda_1 > \lambda_2 > \cdots > \lambda_n >0$, and define $\delta_s > 0$ as $\delta_s := \min_{j : j \neq i}|\lambda_i(\widehat{W}) - \lambda_j(\widehat{W})|.$
    % \[\delta_s := \min_{j : j \neq i}|\lambda_i(\widehat{W}) - \lambda_j(\widehat{W})|.\]
    Let $\widetilde{W} = \widehat{W}+E \in \mathbb{R}^{n \times n}$ be a symmetric perturbed ground-truth matrix where $E \in \mathbb{R}^{n \times n}$ is the noisy symmetric matrix with $\lnorm{E}\in[0,\delta_s/2]$.
    Let $\widetilde{W} = \widetilde{V}\widetilde{\Sigma}\widetilde{V}^{\top}$ be an eigendecomposition of $\widetilde{W}$ and $\widehat{W} = V\Sigma V^{\top}$ be an eigendecomposition of $\widehat{W}$. Let $W(k) = W_{N}(k)\cdots W_1(k)$ follow the iteration of gradient descent in (\ref{eq:dp-update}). Let $M := \max\{\alpha, \|\widetilde{W}\|^{\frac{1}{N}}\}$. Assume that 
    \begin{equation*}\label{eta:cond:mthm}
    0 < \eta < \frac{1}{(3N-2) M^{2N-2}}.
    \end{equation*}
    Then $W(k)$ converges to $V\Sigma_{+}V^{\top}$ as $k \to \infty$ where $\Sigma_{+} = diag(\max\{\lambda_i,0\}: i \in [n])$. 
    The diagonal error matrix $\mathcal{E}(k) = (V^{\top}W(k)V - \Sigma_{+})$, whose entries have the following result:
    \begin{equation*}\label{eq:error-estimate}
     |\mathcal{E}_{ii}(k)| \leq \begin{cases} 
     \sbrac{\frac{4\sqrt{2}M^N}{\delta_s}+1}\lnorm{E} + \widetilde{\epsilon}  N \nl_i^{1-\frac{1}{N}} & \mbox{ if } \nl_i >0,\\
    \sbrac{\frac{4\sqrt{2}M^N}{\delta_s}+1}\lnorm{E}+  \widetilde{\epsilon}^N & \mbox{ if } \nl_i \leq 0.
    \end{cases}
    \end{equation*}
for all $k \geq T_{N}^{\Id}(\nl_i, \widetilde{\epsilon},\alpha, \eta)$, where $ T_{N}^{\Id}(\nl_i, \widetilde{\epsilon},\alpha, \eta)$ is defined in \cite[Equation~(22)]{chou2024gradient}.
\end{theorem}
\begin{proof}
The convergence of $\Wk$ towards $V\Sigma_{+}V^{\top}$ satisfies directly from the result of Lemma~\ref{lemma:convergence}. We next derive $\abs{\mathcal{E}_{ii}(k)}$. Let $\widetilde{\Sigma}_i = \text{diag}(\max\{\nl_i,0\}: i \in [n])$ and $\widetilde{\mathcal{E}}=(\widetilde{V}^{\top}W(k)\widetilde{V} - \widetilde{\Sigma}_{+})$. By Lemma~\ref{lemma:dp-dynammics}, $\widetilde{\mathcal{E}}$ is diagonal.
From the result in \cite[Theorem 1.1]{chou2024gradient}, for $\widetilde{\nl_i}$ and $\widetilde{\epsilon}$, we have that
\begin{equation}\label{eq:dp-error-est}
    |\widetilde{\mathcal{E}}_{ii}(k)| \leq 
    \begin{cases} 
    \widetilde{\epsilon}  N \nl_i^{1-\frac{1}{N}} & \mbox{ if } \nl_i >0,\\
    \widetilde{\epsilon}^N & \mbox{ if } \nl_i \leq 0.
    \end{cases}
    \end{equation}
Then, by the triangle inequality for $|\mathcal{E}_{ii}(k)|$, we have
    \begin{align*}
        \left|\mathcal{E}_{ii}(k)\right| &= \left| \left[ V^{\top}W(k)V\right]_{ii} - (\lambda_{+})_i\right|\\
        &\leq \underbrace{\left| \left[ V^{\top}W(k)V\right]_{ii} - \left[ \widetilde{V}^{\top}W(k)\widetilde{V}\right]_{ii}\right|}_{A_1}  + \underbrace{\left| \left[ \widetilde{V}^{\top}W(k)\widetilde{V}\right]_{ii} - (\nl_+)_i\right|}_{A_2} +A_3
    \end{align*}
    where $A_3= |(\nl_+)_i - (\lambda_+)_i|$, which is bounded by $\|E\|$ because $A_3 \leq |\widetilde{\lambda}_i - \lambda_i| \leq \|E\|$.
    $A_2$ is $|\widetilde{\mathcal{E}}_{ii}|$ and its bound is given in (\ref{eq:dp-error-est}). Since $[V^{\top}W(k)V]_{ii} = v_i^{\top}W(k)v_i$ and $[\widetilde{V}^{\top}W(k)\widetilde{V}]_{ii} = \widetilde{v}_i^{\top}W(k)\widetilde{v}_i$, we  can bound $A_1$ by
    \begin{align*}
    A_1
    &= \left|v_i^{\top}W(k)v_i
    - \widetilde{v}_i^{\top}W(k)\widetilde{v}_i \right|\\
    &= \left|(v_i-\widetilde{v}_i)^{\top}W(k)v_i
    + \widetilde{v}_i^{\top}W(k)(v_i-\widetilde{v}_i)\right|\\
    &\leq \|v_i-\widetilde{v}_i\|\,\|W(k)\|\,\|v_i\|
    + \|\widetilde{v}_i\|\,\|W(k)\|\,\|v_i-\widetilde{v}_i\|\\
    &=2\|W(k)\|\|v_i-\widetilde{v}_i\|,
\end{align*}
assuming $\|v_i\| = \|\wt{v}_i\| = 1$.
    % \begin{align*}
    %     A_1& = \left| \left[ V^{\top}W(k)V\right]_{ii} - \left[ \widetilde{V}^{\top}W(k)\widetilde{V}\right]_{ii}\right|\\
    %     &= \left|v_i^{\top}W(k)v_i - \widetilde{v}_i^{\top}W(k)\widetilde{v}_i \right|\\
    %     &= \left|v_i^{\top}W(k)v_i -\widetilde{v}_i^{\top} W(k)v_i + \widetilde{v}_i^{\top} W(k)v_i - \widetilde{v}_i^{\top}W(k)\widetilde{v}_i \right|\\
    %     &\leq \left|v_i^{\top}W(k)v_i -\widetilde{v}_i^{\top} W(k)v_i \right| + \left|\widetilde{v}_i^{\top} W(k)v_i - \widetilde{v}_i^{\top}W(k)\widetilde{v}_i \right| && \text{(Triangle Inequality)}\\
    %     &= \left|(v_i - \widetilde{v}_i)^{\top}W(k)v_i\right|+\left| \widetilde{v}_i^{\top}W(k)(v_i - \widetilde{v}_i) \right|\\
    %     &\leq \|(v_i - \widetilde{v}_i)\|\|W(k)v_i\| + \|\widetilde{v}_i^{\top}\|\|W(k)(v_i-\widetilde{v}_i)\| &&\text{(Cauchy-Schwarz Inequality)}\\
    %     &= \|(v_i - \widetilde{v}_i)\|\|W(k)v_i\|  + \|W(k)(v_i-\widetilde{v}_i)\| &&\text{($\|v_i^{\top}\|=1$)}\\
    %     &\leq \|(v_i-\widetilde{v}_i)\|\|W(k)\|\|v_i\|+ \|W(k)\|\|v_i - \widetilde{v}_i\| &&\text{(Cauchy-Schwarz Inequality)}\\
    %     &=2\|W(k)\|\|v_i - \widetilde{v}_i\|.
    % \end{align*}
    For $\|W(k)\|$, we have that $\|W(k)\| \leq M^N$. We next bound $\|v_i - \widetilde{v}_i\|$. 
    % Recall that $\widehat{W}$ and $\widetilde{W}$ are two symmetric matrices, and define 
    % \begin{equation}
    % \label{eq:dk-sep}
    %     \delta_s = \min_{j : j \neq i}|\lambda_i(\widehat{W}) - \lambda_j(\widehat{W})|
    % \end{equation}
    % for some $\delta_s > 0$. 
    Define the angle between $v_i$ and $\widetilde{v}_i$ as $\theta_i := \arccos\left( \frac{\|\langle v_i(\widehat{W}), \widetilde{v}_i(\widetilde{W})\rangle\|}{\|v_i(\widehat{W})\| \cdot \|\widetilde{v}_i(\widetilde{W}) \|}\right)$. Then, by the Davis-Kahan theorem, $\sin(\theta_i) \leq 
    % \frac{2\|\widehat{W} - \widetilde{W}\|}{\delta_s}
     \frac{2\|E\|}{\delta_s}$. With our assumption on $\lnorm{E} \in [0, \delta_s/2]$.
    % \begin{equation}
    %     \sin(\theta_i) \leq \frac{2\|\widehat{W} - \widetilde{W}\|}{\delta_s} = \frac{2\|E\|}{\delta_s} \leq \frac{2\ebound}{\delta_s}.
    % \end{equation}
    % Then, let 
    % \begin{equation}
    %     0 \leq \sigma \leq \frac{\delta_s}{2C\sbrac{\sqrt{n} + \sqrt{\ln\sbrac{4/\delta'}}}}
    % \end{equation}
    % \begin{equation}
    %     \epsilon \geq \frac{2\sqrt{2}C}{\delta_s}\sbrac{\sqrt{n} + \sqrt{\ln\frac{4}{\delta'}}}\sqrt{\ln\frac{1.25}{\delta}}
    % \end{equation} 
    $\|v_i - \widetilde{v}_i\|$ is as 
    \begin{equation}
    \label{eq:v-bound}
        \|v_i - \widetilde{v}_i\|^2 = \langle v_i - \widetilde{v}_i, v_i - \widetilde{v}_i\rangle = 2 - 2\sqrt{1-(\sin\theta_i)^2}
        \leq 2 - 2\sqrt{1-\frac{4\lnorm{E}^2}{\delta_s^2}}, 
    \end{equation}
    which is well-defined.
    Considering the function $f(x)= \sqrt{1-x}, x \in [0,1]$. It is concave since its second derivative $f''(x) = -\frac{1}{4}\sbrac{1-x}^{-\frac{3}{2}}$ is always negative on $[0,1)$. Then take the chord between $(0,1)$ and $(1,0)$, we have 
    $L(x) = f(0) + \frac{f(1)-f(0)}{1-0}(x-0) = 1-x$. By the chord inequality for concave function, for $x \in [0,1]$,
    \begin{equation}
    \label{eq:chord}
        \sqrt{1-x} \geq 1- x
    \end{equation}
    holds. With \eqref{eq:chord} and \eqref{eq:v-bound}, 
    \begin{equation}
    \label{eq:dk-bound}
        \|v_i -\widetilde{v}_i\| \leq \frac{2\sqrt{2}\lnorm{E}}{\delta_s}.
    \end{equation}
    % \noteJW{8.18: Should we use $\lesssim$ here?}
    For $A_1$, 
    \begin{equation*}
        A_1 \leq 4\sqrt{2}M^N\frac{\lnorm{E}}{\delta_s}.
    \end{equation*}
    With the upper bound of $A_2$ and $A_3$ in hand, we finally upper-bound for $|\mathcal{E}_{ii}(k)|$, 
    \begin{equation*}
        |\mathcal{E}_{ii}(k)| \leq A_1 + A_2 + A_3 \leq \begin{cases} 
     \sbrac{\frac{4\sqrt{2}M^N}{\delta_s}+1}\lnorm{E} + \widetilde{\epsilon}  N \nl_i^{1-\frac{1}{N}} & \mbox{ if } \nl_i >0,\\
    \sbrac{\frac{4\sqrt{2}M^N}{\delta_s}+1}\lnorm{E}+  \widetilde{\epsilon}^N & \mbox{ if } \nl_i \leq 0,
    \end{cases}
    \end{equation*}
    holds. This completes the proof.  
\end{proof}

\section{The Stability of Implicit Regularization}
\label{sec:stability}

We start by demonstrating the implicit regularization results of the gradient descent in \eqref{eq:dp-update} via Theorem~\ref{thm:gd-dp-ib}. Let us denote
\begin{equation}
\label{eq:tilde_T01}
     \wt{T}_0 = T_0(\{\nlrpi{\ell}\}_{\ell=1}^{\L},\epsilon,\alpha,\eta),
     \quad
     \wt{T}_1 =  T_1(\nlrpi{L+1},\epsilon',\alpha,\eta),
\end{equation}
where $T_0$ and $T_1$ are defined in \eqref{eq:T0} and \eqref{eq:T1}, respectively.

\begin{theorem}
\label{thm:gd-dp-ib}
 Let $\widehat{W} \in \mathbb{R}^{n \times n}$ be a symmetric ground-truth matrix with eigenvalues $\lambda_1 \geq \lambda_2 \geq \cdots \lambda_n \geq 0$  and $\widetilde{W} = \hW + E \in \mathbb{R}^{n \times n}$ be our perturbed symmetric ground-truth matrix with eigenvalues $\nl_1\geq \nl_2 \geq \cdots \geq \nl_n$, where $E \in \mathbb{R}^{n \times n}$ be a symmetric noise matrix.
 Let matrix $\widetilde{W}^{+} \in \mathbb{R}^{n \times n}$ be a symmetric matrix with the eigenvalues $(\nl_i)_{+}$ such that $(\nl_i)_{+} = \max\{ \nl_i, 0\}$ for all $i \in [n]$.

Assume that $W_1(k),...,W_N(k) \in \mathbb{R}^{n \times n}$ and $W(k) = W_N(k)\cdots W_1(k)$ follow the gradient descent as in \eqref{eq:dp-update} for $N \geq 2$. Fix $L \in [n]$ and let $(\nl_L)_{+} > 0$. We have two error terms $\epsilon \in (0,1)$ and $\epsilon' \in (0,c_N)$ where $c_N = \frac{N-1}{2N-1}$. We assume that $\alpha^N < \epsilon'(\nl_L)_{+}$ and the step size $\eta$ in (\ref{eq:dp-update}) satisfies 
$\eta < 1/[(3N-2) \max\{\alpha^{N-2}, (\nl_1)_+^{2-\frac{2}{N}}\}]$. Define $L' = \max\{\ell \in [n]: \epsilon'(\nl_\ell)_+ > \alpha^N\}$, and $L'' = \max\{\ell \in [n] : (\nl_\ell)_+ > \alpha^N\}$. 
Then, for all $k$ satisfying $k \in [\wt{T}_0, \wt{T}_1]$, we have that 
\begin{align}
\label{eq:effective-rank-bound}
        \left| r(\widehat{W}_L) - r(W(k))\right| &\leq  2L\frac{\|E\|}{\lambda_1} + \epsilon r(\widetilde{W}^+_L) + \frac{2(L'-L)}{c_N}
        \frac{(\nl_{L+1})_+}{(\nl_1)_+} \epsilon'\notag\\
&\quad\quad 
+ (n-L')\frac{2\alpha^N}{\epsilon'(\nl_1)_+}.
\end{align}
where $\wt{T}_0$ and $\wt{T}_1$ are defined in \eqref{eq:tilde_T01}.
Moreover, by setting 
\begin{equation}
\label{eq:simp-bound-setting}
\begin{split}
    \epsilon' &= c_N\min \left\{\frac{(\nl_1)_+}{(\nl_{L+1})_+}\frac{\epsilon r(\widetilde{W}^+_L)}{2(n-L)}, 1\right\},\\
    \alpha &= \min\left\{ \left(\epsilon'(\nl_1)_+\right)^{\frac{1}{N}}\left(\frac{\epsilon r(\widetilde{W}^+_L)}{2(n-L)} \right)^{\frac{1}{N}}, \left(\epsilon' (\nl_{L+1})_{+}  \right)^{\frac{1}{N}}\right\}
\end{split}
\end{equation}
we can simplify (\ref{eq:effective-rank-bound}) as
\begin{align}
\label{eq:effective-rank-bound-sim}
        \left| r(\widehat{W}_L) - r(W(k))\right| \leq  2L\frac{\|E\|}{\lambda_1} + 3\epsilon r(\widetilde{W}_L^{+}).
\end{align}
where $\Wstar_\L$ and $\widetilde{W}^+_L$ are the best rank $\L$ approximations of $\Wstar$ and $\widetilde{W}^+$, respectively and $r(\W) = \frac{\|\W\|_*}{\|\W\|}$ is the effective rank.
\end{theorem}
\begin{proof}
     By the triangle inequality, we have
    \begin{equation}
    \label{eq:eff-rank-noise}
        \left| r(\widehat{W}_L) - r(W(k))\right| \leq  \underbrace{\left| r(\widehat{W}_L) -r(\widetilde{W}^{+}_L)\right|}_{:A_1}+ \underbrace{\left| r(\widetilde{W}^{+}_L) - r(W(k))\right|}_{:A_2}
        % \underbrace{\left| r(\overline{W}(k)) - r(W(k))\right|}_{:A_3}.
    \end{equation}
    By definition of the effective rank, we have
    \begin{align*}
        A_1 &=
        \left| \frac{||\hW_L||_*}{\lambda_1} - \frac{||\wt{W}^+_L||_*}{(\nl_1)_+}\right| 
        \leq \underbrace{\left| \frac{||\hW_L||_*}{\lambda_1} - \frac{||\wt{W}^+_L||_*}{\lambda_1} \right|}_{A_{11}} + \underbrace{\left|\frac{||\wt{W}^+_L||_*}{\lambda_1} - \frac{||\wt{W}^+_L||_*}{(\nl_1)_{+}}\right|}_{A_{12}}. \\
    \end{align*}
    \textbf{Bounding $A_{11}$}: We first consider bounding $|\lambda_i - (\nl_i)_{+}|$. If $\nl_i \geq 0$, we have 
    $|\lambda_i - (\nl_i)_{+}| = |\lambda_i - \nl_i| \leq \|E\|$. If $\nl_i <0$, $|\lambda_i - (\nl_i)_{+}| < |\lambda_i - \nl_i|$ holds. So for the above, we have that
    \[\vert\lambda_i - (\nl_i)_{+}\rvert \leq  \vert\lambda_i - \nl_i\rvert \leq \|E\|.\]Next, by the triangle inequality , we have 
    \begin{align*}
        A_{11} \leq \frac{1}{\lambda_1}\sum_{i=1}^{L}\lvert\lambda_i - (\nl_i)_{+}\rvert \leq \frac{1}{\lambda_1}\sum_{i=1}^{L}\lvert\lambda_i - \nl_i\rvert  \leq \frac{L}{\lambda_1}\|E\|.
    \end{align*}
    \textbf{Bounding $A_{12}$:} Since
    \begin{align*}
        \left|\frac{1}{\lambda_1} - \frac{1}{(\nl_1)_{+}}\right|
        &= \left|\frac{(\nl_1)_{+} - \lambda_1}{\lambda_1(\nl_1)_{+}} \right|
        = \frac{\lvert(\nl_1)_{+} - \lambda_1\rvert}{\lambda_1(\nl_1)_{+}}
        \leq  \frac{\lvert\nl_1 - \lambda_1\rvert}{\lambda_1(\nl_1)_{+}}
        \leq \frac{\|E\|}{\lambda_1(\nl_1)_{+}}
    \end{align*}
    and $A_{12}=|\frac{1}{\lambda_1} - \frac{1}{(\nl_1)_{+}}|\sum_{i=1}^{L}(\nl_i)_{+}$, we can upper-bound $A_1$ as
    \begin{equation}
    \label{eq:ib-noise-A1}
        A_1 \leq A_{11} + A_{12} \leq \frac{L}{\lambda_1}\|E\| + \frac{\|E\|}{\lambda_1(\nl_1)_{+}}\sum_{i=1}^{L}(\nl_i)_{+} = \frac{\|E\|}{\lambda_1}\left(L+ \frac{1}{(\nl_1)_{+} } \sum_{i=1}^{L}(\nl_i)_{+} \right)
    \end{equation}
For $A_2$, by \cite[Theorem 3.5]{chou2024gradient}, we have that
\begin{equation}
\label{eq:ib-noise-A2}
     A_2 \leq \epsilon r(\widetilde{W}^+_L) + \frac{2(L'-L)}{c_N}
        \frac{(\nl_{L+1})_+}{(\nl_1)_+} \epsilon'
        +        (n-L')\frac{2\alpha^N}{ \epsilon'(\nl_1)_+}.
\end{equation}
    Therefore, combining \eqref{eq:ib-noise-A1}, \eqref{eq:ib-noise-A2}, and \eqref{eq:eff-rank-noise},
    \
    \begin{align}
        \left| r(\widehat{W}_L) - r(W(k))\right|
        &\leq \frac{\|E\|}{\lambda_1}\left(L+ \frac{1}{(\nl_1)_{+} } \sum_{i=1}^{L}(\nl_1)_{+} \right) + \epsilon r(\widetilde{W}^+_L)\notag\\
        &\quad+ \frac{2(L'-L)}{c_N}
        \frac{(\nl_{L+1})_+}{(\nl_1)_+} \epsilon'
        +        (n-L')\frac{2\alpha^N}{ \epsilon'(\nl_1)_+}\notag\\
        &\leq 2L\frac{\|E\|}{\lambda_1}+\epsilon r(\widetilde{W}^+_L) + \frac{2(L'-L)}{c_N}
        \frac{(\nl_{L+1})_+}{(\nl_1)_+} \epsilon'\notag \\
        &\quad+        (n-L')\frac{2\alpha^N}{ \epsilon'(\nl_1)_+},
        \label{eq:A1A2-bound}
    \end{align}
where the second inequality is due to $(\nl_1)_{+} > (\nl_2)_{+} > \cdots >(\nl_n)_{+} > 0$.
Moreover, by setting $\epsilon'$ and $\alpha$ shown in (\ref{eq:simp-bound-setting}),
$A_2$ can be simplified as:
\begin{align*}
    A_2 &\leq \epsilon r(\widetilde{W}^+_L) + \frac{2(L'-L)}{c_N}
        \frac{(\nl_{L+1})_+}{(\nl_1)_+} \epsilon'
        +        (n-L')\frac{2\alpha^N}{ \epsilon'(\nl_1)_+}\\
        &\leq \epsilon r(\widetilde{W}^+_L) + \frac{2(n-L)}{c_N}
        \frac{(\nl_{L+1})_+}{(\nl_1)_+} \epsilon'
        +        (n-L)\frac{2\alpha^N}{ \epsilon'(\nl_1)_+}\\
        &\leq \epsilon r(\widetilde{W}^+_L) + \frac{2(n-L)}{c_N}
        \frac{(\nl_{L+1})_+}{(\nl_1)_+} \left(c_N\frac{(\nl_1)_+}{(\nl_{L+1})_+}\frac{\epsilon r(\widetilde{W}^+_L)}{2(n-L)}\right)\\
        &\quad+  (n-L)\frac{2}{ \epsilon'(\nl_1)_+} \left(\epsilon'(\nl_1)_+\right)\left(\frac{\epsilon r(\widetilde{W}^+_L)}{2(n-L)} \right)\\
        &= 3\epsilon r(\widetilde{W}^+_L),
\end{align*}
where the second inequality holds since $L'-L \leq n-L$,\,$n-L' \leq n-L$. With \eqref{eq:A1A2-bound}, \eqref{eq:effective-rank-bound-sim} holds. This completes the proof.
\end{proof}
Theorem~\ref{thm:gd-dp-ib} establishes implicit regularization results for the gradient descent dynamics in \eqref{eq:dp-update}. Roughly speaking, it characterizes how the introduction of the noise matrix $E$ affects the upper bound on $\lvert r(\widehat{W}_L)-r(W(k))\rvert$ when $k \in [\wt{T}_0,\wt{T}_1]$.

Recall the spectral conditions in Theorem~\ref{thm:spectral-time}. If the eigenvalues of $\widetilde{W}^+$ satisfy the corresponding conditions with respect to $\widetilde{W}^+$, then $\wt{T}_0 < \wt{T}_1$, as defined in \eqref{eq:tilde_T01}. Suppose also that $T_0 < T_1$ holds under Theorem~\ref{thm:spectral-time}. We next present Theorem~\ref{thm:noise-ib-time}, which quantifies how the noise matrix $E$ shifts the plateau interval $[T_0,T_1]$.
For convenience, let us first denote
\begin{equation}
\label{eq:B-C}
    H(\lambda) = \ln\sbrac{\frac{\lambda}{\alpha^2} - 1}, \quad C_{\epsilon'} = \ln\sbrac{\frac{1}{\epsilon'}-1}, \quad C_2 = \ln 2.
\end{equation}
% \begin{itemize}
%     \item Let 5.1 hold with $N=2$ and consider the same setting.
%     \item Mention $k \in [T_0, T_1]$. under which what holds.
% \end{itemize}
\begin{theorem}
\label{thm:noise-ib-time}
% Fix $N=2$. Let $\widehat{W} \in \mathbb{R}^{n \times n}$ be the symmetric ground-truth matrix with eigenvalues $\lambda_1 > \lambda_2 > \cdots \lambda_n > 0$. Let $\widetilde{W} = \widehat{W} + E$ be a symmetric perturbation defined in \eqref{eq:dp-prob} with eigenvalues  $\nl_1 > \nl_2 > \cdots > \nl_n$ and define the positive-part of $\wt{W}$ as $\wt{W}^+$ whose eigenvalues are $\nlrpi{1} > \nlrpi{2} > \cdots > \nlrpi{n} > 0$ where $\nlrpi{i} := \max\{\wt{\lambda_i}, 0\}$.
% Let $\widetilde{W}^{+} \in \mathbb{R}^{n\times n}$ be as above with the eigenvalues $(\nl_1)_{+} > (\nl_2)_{+} > \cdots >(\nl_n)_{+} \geq 0$. 
Consider the same setting as in Theorem~\ref{thm:gd-dp-ib} with $N=2$. For $k \in [\wt{T}_0, \wt{T}_1]$, where $\wt{T}_0$ and $\wt{T}_1$ are defined in \eqref{eq:tilde_T01}, the gradient descent \eqref{eq:dp-update} satisfies the low rank approximation as shown in \eqref{eq:effective-rank-bound}. Let $T_0$ and $T_1$ be defined as in \eqref{eq:T0} and \eqref{eq:T1} such that \eqref{eq:eff-rank-N-2} holds when $k \in [T_0, T_1]$.
% Fix $L \in [n]$. Let $\epsilon \in \sbrac{0,1}$ and $\epsilon' \in \sbrac{0,\frac{1}{3}}$. Let $W(k)=W_2(k)W_1(k) \in \mathbb{R}^{n\times n}$ follow the gradient descent described in \eqref{eq:ori-init}\eqref{eq:ori-update}. Let $\wt{W}(k)= \wt{W_2}(k)\wt{W_1}(k) \in \mathbb{R}^{+}$ follow the gradient descent described in \eqref{eq:dp-2-init}\eqref{eq:dp-2-update}. Let the initialization parameter $\alpha$ in \eqref{eq:ori-init} and \eqref{eq:dp-2-init} satisfy $\alpha^2 \leq \epsilon' \min\left\{\lambda_{L+1}, \nlrpi{L+1}\right\} $ and let the step size $\eta$ in \eqref{eq:ori-update} and \eqref{eq:dp-2-update} satisfy $\eta < \min\left\{\frac{1}{4\lambda_1}, \frac{1}{4\nlrpi{L+1}}\right\}$. Let $T_s$~(defined in \eqref{eq:Tl}) be the time required to approximation the leading $L$ eigenvalues $\{\lambda_i\}_{i \in [n]}$, and let $T_e$~(defined in \eqref{eq:Ts}) be the time after which the remaining eigenvalues of $W(k)$ remain small. Let $\wt{T}_e$ and $\wt{T}_s$ denote the corresponding quantities defined using $\{\nlrpi{i}\}_{i \in [n]}$. 
Assume further that the perturbation magnitude $\lnorm{E}$ satisfies
\begin{equation*}
    \lambda_{L+1} - \lnorm{E} > \alpha^2, \quad 0 < \frac{2}{3}\eta\sbrac{\lambda_{L+1} + \lnorm{E}} < 1.
\end{equation*}
% $\lambda_{L+1} - \lnorm{E} > \alpha^2$, $0 < \frac{2}{3}\eta\sbrac{\lambda_{L+1} + \lnorm{E}} < 1$, and $3-2\eta\sbrac{\lambda_{L+1}+\lnorm{E}} \neq 0$.
Let $K_\epsilon$ be a constant defined in \eqref{eq:K} and let $H(\lambda_{L+1})$, $C_{\epsilon'}$, and $C_2$ be quantities as in \eqref{eq:B-C}.
    % \begin{itemize}
    %     % \item GD iteration.
    %     \item  Let $T_e$ and $\widetilde{T}_e$ be. Let $T_s$ and $\widetilde{T}_s$ be. Let $\lambda_{L+1}, \nlrpi{L+1} \geq \alpha^2$ and $\lambda_{L+1} - \lnorm{E} > \alpha^2$.
    %     \item $\eta$, $\alpha$.
    %     \item $\nlrpi{L+1}$ condition
    %     \item $K_\epsilon$.
    % \end{itemize}
Then, 
% with probability at least $1 - \delta'$, 
    \begin{equation*}
        \abs{\widetilde{T}_1 - T_1} < \frac{\lnorm{E}}{2\eta \sbrac{\lambda_{L+1} - \lnorm{E}}}\sbrac{\frac{1}{\lambda_{L+1} - \lnorm{E} - \alpha^2} + \frac{\abs{\sbrac{H\sbrac{\lambda_{L+1}} - C_{\epsilon'}}}}{\lambda_{L+1}}}, 
    \end{equation*}
    \begin{align*}
     \abs{\widetilde{T}_0 - T_0} &< \frac{\lnorm{E}}{2\eta \sbrac{\lambda_{L+1} - \lnorm{E}}}\sbrac{\frac{1}{\lambda_{L+1} - \lnorm{E} - \alpha^2} + \frac{\abs{\sbrac{H\sbrac{\lambda_{L+1}} - C_{2}}}}{\lambda_{L+1}}}\\
                                &\quad\quad+ \frac{\lnorm{E}}{\sqrt{3}\alpha} \frac{1}{\sqrt{\lambda_{L+1}-\lnorm{E}}+\sqrt{\lambda_{L+1}}}\\
                                &\quad\quad+ \frac{9K_{\epsilon}\lnorm{E}}{2\eta\lambda_{L+1}\sbrac{\lambda_{L+1}-\lnorm{E}}\abs{3-2\eta\sbrac{\lambda_{L+1}+\lnorm{E}}}}+1.
\end{align*}
\end{theorem}
\begin{proof}
    \textbf{Analysis for $T_1$:}~With \eqref{eq:T1} and \eqref{eq:B-C} and the triangle inequality, 
    \begin{equation}
    \label{eq:abs-Te-bound}
         \abs{\widetilde{T}_1 - T_1}\leq \underbrace{\abs{\frac{H\sbrac{\nlrpi{L+1}}-H\sbrac{\lambda_{L+1}}}{2\eta\nlrpi{L+1}}}}_{:=I_1}
         + \underbrace{\abs{\frac{H\sbrac{\lambda_{L+1}}-C_{\epsilon'}}{2\eta}\sbrac{\frac{1}{\nlrpi{L+1}}-\frac{1}{\lambda_{L+1}}}}}_{:=I_2}.
    \end{equation}
    % We next bound $I_1$ and $I_2$.\\
    \textbf{Bound $I_1$:}~Since  
    \begin{equation}
    \label{eq:lambda-abs-bound}
        \abs{\lambda_{L+1} - \nlrpi{L+1}} \leq \abs{\lambda_{L+1} - \nl_{L+1}} \leq \lnorm{E},
    \end{equation}
    % $\abs{\lambda_{L+1} - \nlrpi{L+1}} \leq \lnorm{E}$~(\ref{})
    $\lambda_{L+1}$ and $\nlrpi{L+1}$ is bounded as 
    \begin{equation}
    \label{eq:lambda-single-bound}
        \lambda_{L+1}, \nlrpi{L+1} \in \left[ \lambda_{L+1} - \lnorm{E} , \lambda_{L+1} + \lnorm{E}\right].
    \end{equation}    
    With our assumptions that $\lambda_{L+1} - \lnorm{E} > \alpha^2$, $H(\lambda_{L+1})$ and $H(\nlrpi{L+1})$ are well defined.
    % Consider the function $\ln x$ for $x \in(0,\infty)$, by the mean value theorem, there exists $\xi \in \sbrac{x,y}$, 
    % \begin{equation*}
    %     \abs{\ln x - \ln y} = \frac{1}{\abs{\xi}}\abs{x-y}
    % \end{equation*} 
    % for $x < y$ and $x, y \in \sbrac{0,\infty}$. 
    % \begin{equation}
    %     \lambda_{L+1}, \nlrpi{L+1} \in \left[ \nlrpi{L+1} - \lnorm{E} ,    \nlrpi{L+1} + \lnorm{E}\right].
    % \end{equation}    
   \\
    % \begin{equation*}
    %     \nlrpi{L+1} - \lnorm{E} \leq \lambda_{L+1} \leq \nlrpi{L+1} + \lnorm{E},
    % \end{equation*}
    % \begin{equation*}
    %     \nlrpi{L+1} - \lnorm{E} \leq \nlrpi{L+1} \leq \nlrpi{L+1} + \lnorm{E}.
    % \end{equation*}
    Let $\xi \in \sbrac{\min\left\{\frac{\nlrpi{L+1}}{\alpha^2}-1, \frac{\lambda_{L+1}}{\alpha^2}-1  \right\}, \max\left\{\frac{\nlrpi{L+1}}{\alpha^2}-1, \frac{\lambda_{L+1}}{\alpha^2}-1  \right\}}$, by the mean value theorem, $\xi > \frac{\lambda_{L+1}-\lnorm{E} - \alpha^2}{\alpha^2}$, and \eqref{eq:lambda-abs-bound}
    \begin{equation}
    \label{eq:B-bound}
         \abs{H\sbrac{\nlrpi{L+1}}-H\sbrac{\lambda_{L+1}}} 
         % = \abs{\ln \sbrac{\frac{\nlrpi{L+1}}{\alpha^2}-1} - \ln \sbrac{\frac{\lambda_{L+1}}{\alpha^2}-1}} 
         = \abs{\frac{1}{\alpha^2\xi}}\abs{\nlrpi{L+1} - \lambda_{L+1}} <  \frac{\lnorm{E}}{\lambda_{L+1}-\lnorm{E} - \alpha^2}
    \end{equation}    
    % Since $\xi > \frac{\lambda_{L+1}-\lnorm{E} - \alpha^2}{\alpha^2}$ and with 
    % \begin{align*}
    %     \abs{B\sbrac{\nlrpi{L+1}}-B\sbrac{\lambda_{L+1}}} &= \abs{\ln \sbrac{\frac{\nlrpi{L+1}}{\alpha^2}-1} - \ln \sbrac{\frac{\lambda_{L+1}}{\alpha^2}-1}}\\
    %     &=\abs{\frac{1}{\alpha^2\xi}}\abs{\nlrpi{L+1} - \lambda_{L+1}}\\
    %     &\leq \frac{\abs{\nlrpi{L+1}-\lambda_{L+1}}}{\abs{\nlrpi{L+1}-\lnorm{E} - \alpha^2}}\\
    %     &\leq  \frac{\lnorm{E}}{\abs{\nlrpi{L+1}-\lnorm{E} - \alpha^2}}
    % \end{align*}
    % where the third inequality holds because $\abs{\frac{1}{\xi}} \leq \abs{\frac{1}{\xi_{min}}}$, where $\abs{\xi_{min}} = \abs{\frac{\nlrpi{L+1} - \lnorm{E}-\alpha^2}{\alpha^2}}$.\\
    Combing \eqref{eq:B-bound}, and \eqref{eq:lambda-single-bound}
    \begin{equation}
    \label{eq:I_1-bound}
        I_1 < \frac{\lnorm{E}}{\sbrac{\lambda_{L+1}-\lnorm{E} - \alpha^2} \sbrac{\lambda_{L+1} - \lnorm{E}}}.
    \end{equation}
    % \begin{align*}
    %     I_1 \leq \frac{\lnorm{E}}{\abs{\nlrpi{L+1}-\lnorm{E} - \alpha^2}\abs{\nlrpi{L+1}}} \leq \frac{\lnorm{E}}{\abs{\nlrpi{L+1}-\lnorm{E} - \alpha^2}\abs{\nlrpi{L+1}-\lnorm{E}}}
    % \end{align*}
\textbf{Bound $I_2$:}With \eqref{eq:lambda-abs-bound} and \eqref{eq:lambda-single-bound},
\begin{equation}
\label{eq:I_2-bound}
     I_2 = \abs{\sbrac{H\sbrac{\lambda_{L+1}} - C_{\epsilon'}}}\abs{\frac{1}{\nlrpi{L+1}}-\frac{1}{\lambda_{L+1}}} \leq \abs{\sbrac{H\sbrac{\lambda_{L+1}} - C_{\epsilon'}}}\frac{\lnorm{E}}{\lambda_{L+1}\sbrac{\lambda_{L+1}-\lnorm{E}}}
\end{equation}
% \begin{equation}
% \label{eq:I_2-bound}
%     I_2 &= \abs{\sbrac{B\sbrac{\lambda_{L+1}} - C_{\epsilon'}}}\abs{\frac{1}{\nlrpi{L+1}}-\frac{1}{\lambda_{L+1}}}\\
%         &= \abs{\sbrac{B\sbrac{\lambda_{L+1}} - C_{\epsilon'}}}\frac{\abs{\lambda_{L+1} - \nlrpi{L+1}}}{\abs{\nlrpi{L+1}\lambda_{L+1}}}\\
%         &\leq \abs{B\sbrac{\nlrpi{L+1}+\lnorm{E}}-C_{\epsilon'}}\frac{\lnorm{E}}{\abs{\nlrpi{L+1}\sbrac{\nlrpi{L+1}-\lnorm{E}}}}.
% \end{eqaution}
With \eqref{eq:abs-Te-bound}, \eqref{eq:I_1-bound}, and \eqref{eq:I_2-bound}, 
\begin{equation}
\label{eq:t-T1-bound}
    \abs{\widetilde{T}_1 - T_1} < \frac{\lnorm{E}}{2\eta \sbrac{\lambda_{L+1} - \lnorm{E}}}\sbrac{\frac{1}{\lambda_{L+1} - \lnorm{E} - \alpha^2} + \frac{\abs{\sbrac{H\sbrac{\lambda_{L+1}} - C_{\epsilon'}}}}{\lambda_{L+1}}}
\end{equation}
% \begin{align}
%     \abs{\widetilde{T}_e - T_e} & \leq \frac{1}{2\eta} \sbrac{I_1 + I_2} \notag\\
%     &\leq \frac{1}{2\eta}\sbrac{\frac{\lnorm{E}}{\abs{\nlrpi{L+1}-\lnorm{E} - \alpha^2}\abs{\nlrpi{L+1}-\lnorm{E}}}}\notag \\
%     & \quad +\frac{1}{2\eta}\sbrac{\abs{B\sbrac{\nlrpi{L+1}+\lnorm{E}}-C_{\epsilon'}}\frac{\lnorm{E}}{\abs{\nlrpi{L+1}\sbrac{\nlrpi{L+1}-\lnorm{E}}}}}\label{eq: TxD}
% \end{align}
\textbf{Analysis for $T_0$}
Next, recall from \eqref{eq: Tx}
\begin{equation*}
        T(\lambda) = \underbrace{\frac{1}{\eta} \sbrac{-\frac{1}{2\lambda}\ln2 +\frac{1}{2\lambda}\ln\sbrac{\frac{\lambda}{\alpha^2}-1}}}_{:=A(\lambda)} + \underbrace{\ceil*{\sbrac{\frac{\lambda}{3}}^{\frac{1}{2}}\frac{1}{\alpha}}}_{:=B(\lambda)} + \underbrace{\frac{\ln\frac{8}{\epsilon}-\abs{\ln \sbrac{1 - \sbrac{\frac{1}{3}}^{\frac{1}{2}}}}}{\abs{\ln \sbrac{1 - \frac{2}{3}\eta \lambda}}}}_{:C(\lambda)}
\end{equation*}
and let $C_2 = \ln 2$ be a constant, then, with the triangle inequality, 
\begin{equation}
\label{eq:T_s-diff}
\begin{split}
    \abs{T_0 - \widetilde{T}_0}
    &\leq \underbrace{\abs{A\sbrac{\nlrpi{L+1} }
    -A\sbrac{\lambda_{L+1}}}}_{:I_A} 
    + \underbrace{\abs{B\sbrac{\nlrpi{L+1} }- B\sbrac{\lambda_{L+1} }}}_{:I_B}\\
    &\quad+ \underbrace{\abs{C\sbrac{\nlrpi{L+1}}-C\sbrac{\lambda_{L+1}}}}_{:=I_C}
\end{split}
\end{equation}
% \begin{align*}
%     \abs{T\sbrac{\nlrpi{L+1}} - T\sbrac{\lambda_{L+1}}} &= \abs{A\sbrac{\nlrpi{L+1} }-A\sbrac{\lambda_{L+1} }+ E\sbrac{\nlrpi{L+1} }- E\sbrac{\lambda_{L+1} }+ D\sbrac{\nlrpi{L+1}}-D\sbrac{\lambda_{L+1}}}\\
%     &\leq \underbrace{\abs{A\sbrac{\nlrpi{L+1} }-A\sbrac{\lambda_{L+1}}}}_{:I_A} + \underbrace{\abs{E\sbrac{\nlrpi{L+1} }- E\sbrac{\lambda_{L+1} }}}_{:I_E}+ \underbrace{\abs{D\sbrac{\nlrpi{L+1}}-D\sbrac{\lambda_{L+1}}}}_{:=I_D}.
% \end{align*}
For $\abs{A\sbrac{\nlrpi{L+1} }-A\sbrac{\lambda_{L+1}}}$, we can replace $C_{\epsilon'}$ as $C_2$, and apply the results in \eqref{eq:t-T1-bound} to $I_A$,
\begin{equation}
\label{eq:I_A}
    I_A <  \frac{\lnorm{E}}{2\eta \sbrac{\lambda_{L+1} - \lnorm{E}}}\sbrac{\frac{1}{\lambda_{L+1} - \lnorm{E} - \alpha^2} + \frac{\abs{\sbrac{H\sbrac{\lambda_{L+1}} - C_{2}}}}{\lambda_{L+1}}}
\end{equation}
% \begin{align*}
%     I_A &\leq \frac{1}{2\eta}\sbrac{\frac{\lnorm{E}}{\abs{\nlrpi{L+1}-\lnorm{E} - \alpha^2}\abs{\nlrpi{L+1}-\lnorm{E}}}}\notag \\
%     & \quad +\frac{1}{2\eta}\sbrac{\abs{B\sbrac{\nlrpi{L+1}+\lnorm{E}}-C_{2}}\frac{\lnorm{E}}{\abs{\nlrpi{L+1}\sbrac{\nlrpi{L+1}-\lnorm{E}}}}} 
% \end{align*}
% $x, y \in \mathbb{R}$, $ \abs{\ceil*{x} - \ceil*{y}} \leq 2\abs{x-y}$
\textbf{Bound $I_B$:} We have that $x, y \in \mathbb{R}$, $ \abs{\ceil*{x} - \ceil*{y}} \leq \abs{x-y}+1$, together with \eqref{eq:lambda-abs-bound} and \eqref{eq:lambda-single-bound}, 
\begin{equation}
\label{eq:I_E}
     I_B = \abs{\ceil*{\sqrt{\frac{\nlrpi{L+1}}{3}}\frac{1}{\alpha}} - \ceil*{\sqrt{\frac{\lambda_{L+1}}{3}}\frac{1}{\alpha}}} \leq \frac{\lnorm{E}}{\sqrt{3}\alpha} \frac{1}{\sqrt{\lambda_{L+1}-\lnorm{E}}+\sqrt{\lambda_{L+1}}}+1.
\end{equation}
\textbf{Bound $I_C$:}
\begin{equation}
\label{eq:I_D}
    I_C = K_{\epsilon} \abs{\frac{1}{\ln\sbrac{1-\frac{2}{3}\eta\lambda_{L+1}}} - \frac{1}{\ln(1-\frac{2}{3}\eta\nlrpi{L+1})}} = K_{\epsilon} \frac{\abs{v-u}}{\abs{v}\abs{u}},
\end{equation}
where
\begin{equation*}
    u: = \ln\sbrac{1-\frac{2}{3}\eta\lambda_{L+1}}, \quad\quad v:= \ln\sbrac{1-\frac{2}{3}\eta\nlrpi{L+1}}.
\end{equation*}
To bound $\abs{u}$ and $\abs{v}$, with $0 < \frac{2}{3}\eta\nlrpi{L+1}< \frac{\nlrpi{L+1}}{6\lambda_1} < 1$, we have
\begin{equation}
\label{eq:u-v-single}
    \abs{u} > \frac{2}{3}\eta \lambda_{L+1},
    \qquad\abs{v} > \frac{2}{3}\eta \nlrpi{L+1}.
\end{equation}
To bound $\abs{v-u}$, by the mean value theorem, 
\begin{equation}
\label{eq:v-u}
    \abs{v-u} = \frac{1}{\abs{\xi'}}\abs{\frac{2}{3}\eta\lambda_{L+1} - \frac{2}{3}\eta \nlrpi{L+1}} < \frac{2\eta \lnorm{E}}{\abs{3-2\eta\sbrac{\lambda_{L+1}+\lnorm{E}}}}
\end{equation}
where $\xi'$ is between $1-\frac{2}{3}\eta\nlrpi{L+1}$ and $ 1-\frac{2}{3}\eta\lambda_{L+1}$.
% $\xi' \in \sbrac{\min\left\{1-\frac{2}{3}\eta\nlrpi{L+1}, 1-\frac{2}{3}\eta\lambda_{L+1}\right\}, \max\left\{ 1- \frac{2}{3}\eta\nlrpi{L+1}, 1- \frac{2}{3}\eta\lambda_{L+1} \right\}}$.
The inequality in \eqref{eq:v-u} is the combination of \eqref{eq:lambda-abs-bound}, \eqref{eq:lambda-single-bound}, and $\xi' >  \abs{ 1 - \frac{2}{3}\eta\sbrac{\lambda_{L+1}+\lnorm{E}}}$. Thus, with \eqref{eq:I_D}, \eqref{eq:v-u}, \eqref{eq:u-v-single}, and $\eqref{eq:K}$
\begin{equation}
    I_C < \frac{9K_{\epsilon}\lnorm{E}}{2\eta\lambda_{L+1}\sbrac{\lambda_{L+1}-\lnorm{E}}\abs{3-2\eta\sbrac{\lambda_{L+1}+\lnorm{E}}}}
\end{equation}
Combing \eqref{eq:I_A}, \eqref{eq:I_E}, and \eqref{eq:I_D}
\begin{align*}
     \abs{\widetilde{T}_0 - T_0} &< \frac{\lnorm{E}}{2\eta \sbrac{\lambda_{L+1} - \lnorm{E}}}\sbrac{\frac{1}{\lambda_{L+1} - \lnorm{E} - \alpha^2} + \frac{\abs{\sbrac{H\sbrac{\lambda_{L+1}} - C_{2}}}}{\lambda_{L+1}}}\\
                                &\quad\quad+ \frac{\lnorm{E}}{\sqrt{3}\alpha} \frac{1}{\sqrt{\lambda_{L+1}-\lnorm{E}}+\sqrt{\lambda_{L+1}}}\\
                                &\quad\quad+ \frac{9K_{\epsilon}\lnorm{E}}{2\eta\lambda_{L+1}\sbrac{\lambda_{L+1}-\lnorm{E}}\abs{3-2\eta\sbrac{\lambda_{L+1}+\lnorm{E}}}}+1.
\end{align*}
This completes the proof.
\end{proof}

Theorems~\ref{thm:gd-dp-ib} and~\ref{thm:noise-ib-time} provide robustness results for the implicit regularization of the gradient descent dynamics in \eqref{eq:dp-update}. Beyond these results, we further show in Theorem~\ref{thm:noise-approx} that, under the same gradient descent dynamics, the matrix approximation error $\|W_L(k)-\hW_L\|$ is also upper-bounded.

\begin{theorem}\label{thm:noise-approx}
Consider the same setting as in Theorem~\ref{thm:gd-dp-ib} with $N\geq2$. For $k \in [\wt{T}_0, \wt{T}_1]$, where $\wt{T}_0$ and $\wt{T}_1$ are defined in \eqref{eq:tilde_T01}, the gradient descent \eqref{eq:dp-update} satisfies the low rank approximation as shown in \eqref{eq:effective-rank-bound}. 
Define $ \delta_s = \min_{j : j \neq i}|\lambda_i - \lambda_j|$.
Then for $k \in [\wt{T}_0, \wt{T}_1]$,
\begin{equation*}
      \fnorm{\Wk_L - \hW_L} < \sbrac{\frac{4\sqrt{2L}}{\delta_s}\lambda_1 + \sqrt{L}}\lnorm{E} + \frac{\epsilon\sqrt{L}(\lambda_1 + \lnorm{E})}{4}.
      % \frac{4\sqrt{2L}\lnorm{E}}{\delta_s}\lambda_1 + \frac{\epsilon\sqrt{L}(\lambda_1 + \lnorm{E})}{4}+\sqrt{L}\lnorm{E}
\end{equation*}
where $W(k)_L$ and $\hW_L$ be the best rank-$L$ approximation of $W(k)$ and $\hW$, repectively.
\end{theorem}
\begin{proof}Let $\widehat{W} = V\Sigma V^\top$ and $\wt{W} = \wt{V}\wt{\Sigma} \wt{V}^\top$ be eigendecomposition of $\Wstar$ and $\wt{W}$, respectively. Let $V_L \in \mathbb{R}^{n \times L}$ denote the matrix whose columns are top-$L$ eigenvectors of $V$, same applies to $\widetilde{V}_L \in \mathbb{R}^{n \times L}$. In addition, let $\Sigma_L := \text{diag}(\lambda_1,...,\lambda_L) \in \mathbb{R}^{L \times L}$ and $\wt{\Sigma}_L := \text{diag}(\wt{\lambda}_1,...,\wt{\lambda}_L) \in \mathbb{R}^{L \times L}$. Decomposing the norm of the difference as
\begin{align*}
        \fnorm{\Wk_L - \hW_L }
        &= \fnorm{\wt{V}_L\wt{\Sigma}(k)_L\wt{V}_L^{\top}- V_L\Sigma_LV_L^{\top} }\\
        &= \fnorm{(\wt{V}_L-V_L)\Sigma_L\wt{V}_L^{\top} + \wt{V}_L(\wt{\Sigma}(k)_L - \Sigma_L)\wt{V}_L^{\top}+ V_L\Sigma_L(\wt{V}_L - V_L)^{\top}}\\
        &\leq \|\wt{V}_L\|\,\|\Sigma_L\|\fnorm{\wt{V}_L-V_L} 
        + \|\wt{V}_L\|\,\|\wt{V}_L\|\fnorm{\wt{\Sigma}(k)_L - \Sigma_L}\\ 
        &\quad+ \lnorm{V_L}\lnorm{\Sigma_L}\fnorm{\wt{V}_L - V_L} \\
        &= 2\lambda_1 \underbrace{\fnorm{\wt{V}_L - V_L}}_{:=A_1} 
        + \underbrace{\fnorm{\wt{\Sigma}(k)_L - \Sigma_L}}_{:=A_2},
\end{align*}
according to the triangle inequality and mixed-norm inequality for Frobenius norm. For $A_1$, by \eqref{eq:dk-bound},
\begin{align*}
    A_1 &= \fnorm{\wt{V}_L - V_L} \
    = \sqrt{\sum_{i=1}^{L}\lnorm{\wt{v}_i - v_i}^2}
    \leq \frac{2\sqrt{2L}\lnorm{E}}{\delta_s}
\end{align*}
For $A_2$, by \eqref{eq:lambda-bound},
\begin{equation*}
    A_2 \leq \sqrt{\sum_{i=1}^{L}\sbrac{\dtil{i} - \nl_i}^2} + \sqrt{\sum_{i=1}^{L}\sbrac{ \nl_i - \lambda_i}^2} \leq \sqrt{L}\max_{1\leq i \leq L}\abs{\dtil{i}-\nl_i}+\sqrt{L}\lnorm{E}
    % \sqrt{L}\,\ebound
\end{equation*}
Thus, 
\begin{align}
\label{eq:WL-bound}
     \fnorm{\Wk_L - \hW_L } \leq \frac{4\sqrt{2L}\lnorm{E}}{\delta_s}\lambda_1 + \sqrt{L}\max_{1\leq i \leq L}\abs{\dtil{i}-\nl_i}+\sqrt{L}\lnorm{E}
\end{align}
Since we assume that $k \in [\wt{T}_0, \wt{T}_1]$, and the initialization $\alpha$ satisfies $\alpha^N < \epsilon'(\nl_{L+1})_+ <\nlrpi{i}$, let $i \in [L]$, then by \cite[Lemma 2.2]{chou2024gradient}, $\dtil{i} \in (0,\nlrpi{i})$ for all $k \in \mathbb{N}_0$, and
\begin{align*}
    \abs{\nlrpi{_i} - d_{\nl_i}^N}=
    \abs{\nlrpi{i}^{\frac{1}{N}} - d_{\nl_i}}\abs{\sum_{i=1}^{N}\nlrpi{i}^{1-\frac{i}{N}}d_{\nl_i}^{i-1}}
    < N\nlrpi{i}^{1-\frac{1}{N}}\abs{\nlrpi{i}^{\frac{1}{N}}-d_{\nl_i}}.
\end{align*}
With $k \geq \wt{T}_0$,
\begin{equation*}
    \abs{\nlrpi{i}^{\frac{1}{N}} - d_{\nl_i}(k)} \leq \frac{(\nl_{i})_+^{\frac{1}{N}}}{4 N} \epsilon, 
\end{equation*}
thus
\begin{equation*}
    \abs{\nlrpi{_i} - \dtil{i}} < N\nlrpi{i}^{1-\frac{1}{N}}\abs{\nlrpi{i}^{\frac{1}{N}}-d_{\nl_i}(k)} \leq \frac{\epsilon \nlrpi{i}}{4}.
\end{equation*}
Note that $\alpha^N <\nlrpi{i}$ for $i \in [L]$, with $[L]$ is finite, $\nlrpi{1} > \nlrpi{2}> \cdots > \nlrpi{L}$, and $ \nlrpi{1} \in \left[ \lambda_{1} - \lnorm{E} , \lambda_{1} + \lnorm{E}\right]$, 
\begin{equation}
\label{eq:max-bound}
    \max_{1\leq i \leq L}\abs{\dtil{i}-\nl_i} = \max_{1\leq i \leq L}\abs{\dtil{i}-\nlrpi{i}} <  \frac{\epsilon(\lambda_1 + \lnorm{E})}{4}
\end{equation} 
Therefore, combining \eqref{eq:WL-bound} and \eqref{eq:max-bound}, 
\begin{equation*}
      \fnorm{\Wk_L - \hW_L} < \sbrac{\frac{4\sqrt{2L}}{\delta_s}\lambda_1 + \sqrt{L}}\lnorm{E} + \frac{\epsilon\sqrt{L}(\lambda_1 + \lnorm{E})}{4}.
\end{equation*}
This completes the proof.
\end{proof}

\section{Numerical Illustrations}
\label{sec:numerics}

We provide numerical illustrations of Theorem~\ref{thm:spectral-time}. We provide detailed analysis for Figure~\ref{fig:ib-vary-eta}; a similar argument can be used for Figure~\ref{fig:ib-vary-lambda} and we omit detail here. In Figure~\ref{fig:ib-vary-eta}, we observe that when $\eta = 0.005$, the three plateaus within which \eqref{eq:gd} approaches rank-1, rank-2, and rank-3 approximations of $\Wstar$ are clearly visible.
Specifically, let $\lambda_1 = 10$, $\lambda_2 = 5$, and $\lambda_3 = 1$, $\alpha = 0.01$ and $\eta = 0.005$. Choose $\epsilon = 0.05$ and $\epsilon' = 0.1$. By \eqref{eq:K}, we have $K_\epsilon \approx 4.2140$. For rank-$1$ approximation, note that
    % \textbf{}~$K_\epsilon \approx 4.2140$. $z = 183$ and $m=53$. $\kappa_{183,53} = 12.31$, so $\frac{m_{12}}{\kappa_{183,53}}\approx4.30$.
$\lambda_1 - \lambda_2 = 5$. By the right-hand side of \eqref{eq:x_gap3}, the required lower bound of $\lambda_1 - \lambda_2$ is approximately $4.3050$. Hence, $\lambda_1 - \lambda_2 > 4.3050$ and \eqref{eq:x_gap3} holds. Moreover, we have $\alpha_*^2 \approx 0.0018 > \alpha^2 = 0.0001$ and $\eta_* \approx 0.009 > \eta = 0.005$, so that both \eqref{eq:alpha-bound} and \eqref{eq:eta-bound} hold. Therefore, the first plateau appears. Similar analysis applies for rank-$2$ and rank-$3$ approximations.

Staying on Figure~\ref{fig:ib-vary-eta}, when the step-size is increased to $\eta = 0.1$, the three plateaus disappear. In this case, the lower-bound in \eqref{eq:x_gap3} is approximately $154$. Under our parameter choice, however, $\lambda_1 - \lambda_2 = 5 < 154$ and thus the explicit expressions for $T_0$ and $T_1$ no longer apply. In addition, the condition $\eta < \eta_*=\min\{0.025, 0.009\}$ is violated since $\eta = 0.1 > 0.009$. Theses two violations explain the disappearance of the plateaus.
\begin{figure}[htbp]
    \centering
    \begin{subfigure}[t]{0.48\textwidth}
        \centering
        \includegraphics[width=\linewidth]{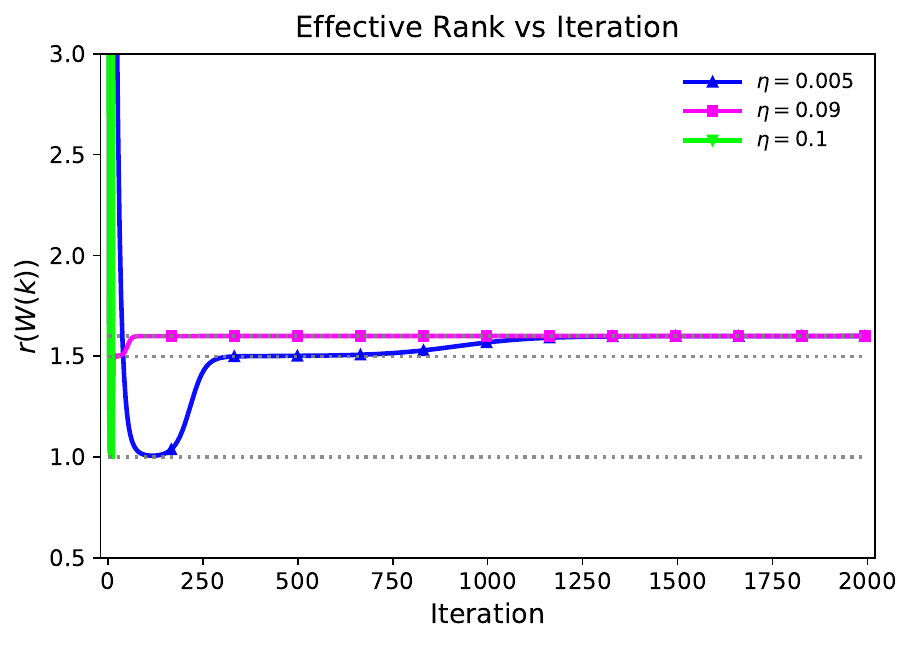}
        % \caption{$N=2$, $n=20$, $\alpha = 10^{-2}, \lambda_1 = 10, \lambda_2 = 5, \lambda_3 = 1$}
        \caption{Varying Step Size}
         \label{fig:ib-vary-eta}
    \end{subfigure}
     \begin{subfigure}[t]{0.48\textwidth}
        \centering
        \includegraphics[width=\linewidth]{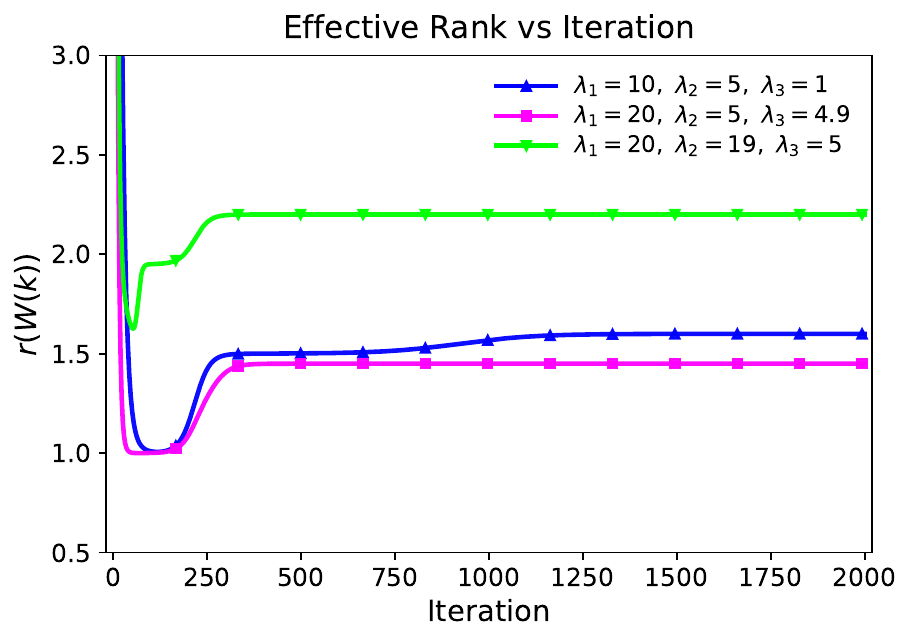}
        % \caption{$N=2$, $n=20$, $\alpha = 10^{-2}$, $\eta = 0.005$}
        \caption{Varying Leading Eigenvalues}
        \label{fig:ib-vary-lambda}
    \end{subfigure}
    \caption{Observability of Implicit Regularization Induced by \eqref{eq:gd}.
   We set $N = 2$, $n = 20$, $\alpha = 10^{-2}$, $\epsilon = 0.05$, and $\epsilon' = 0.1$. Figure~\ref{fig:ib-vary-eta} fixes the leading eigenvalues at $\lambda_1 = 10$, $\lambda_2 = 5$, and $\lambda_3 = 1$, and illustrates the observability of implicit regularization for three choices of $\eta$. Figure~\ref{fig:ib-vary-lambda} fixes $\eta = 0.005$ and illustrates the observability of implicit regularization for three choices of the leading eigenvalues.
  }
\label{fig:spectra-exp}
\end{figure}

We provide a numerical illustration in Figure~\ref{fig:noise} for the results of Theorem~\ref{thm:noise-approx} and Theorem~\ref{thm:noise-ib-time}. We vary $\sigma = c/\sqrt{n}$ so as to control the operator norm of the noise matrix $E$. Figure~\ref{fig:approx-vary-noise} shows that the approximation error $\|W(k) - \Wstar\|_F$ increases with the noise level, which is consistent with Theorem~\ref{thm:noise-approx}. Figure~\ref{fig:ib-vary-noise} shows that increasing the noise level shifts the plateaus associated with low-rank approximation. For instance, for the rank-$2$ approximation, the noise magnitude substantially affects the end time $\wt{T}_1$: larger noise causes gradient descent in \eqref{eq:gd} to leave earlier the interval during which it remains close to the rank-$2$ approximation.
\begin{figure}[htbp]
    \centering
     \begin{subfigure}[t]{0.48\textwidth}
        \centering
        \includegraphics[width=\linewidth]{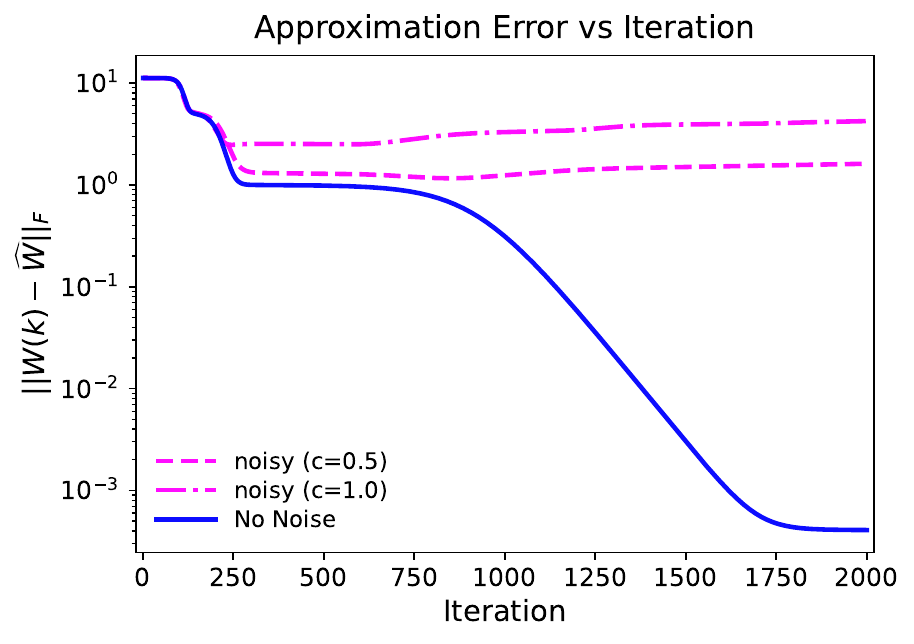}
        \caption{Approximation Error}
         \label{fig:approx-vary-noise}
    \end{subfigure}
    \begin{subfigure}[t]{0.48\textwidth}
        \centering
        \includegraphics[width=\linewidth]{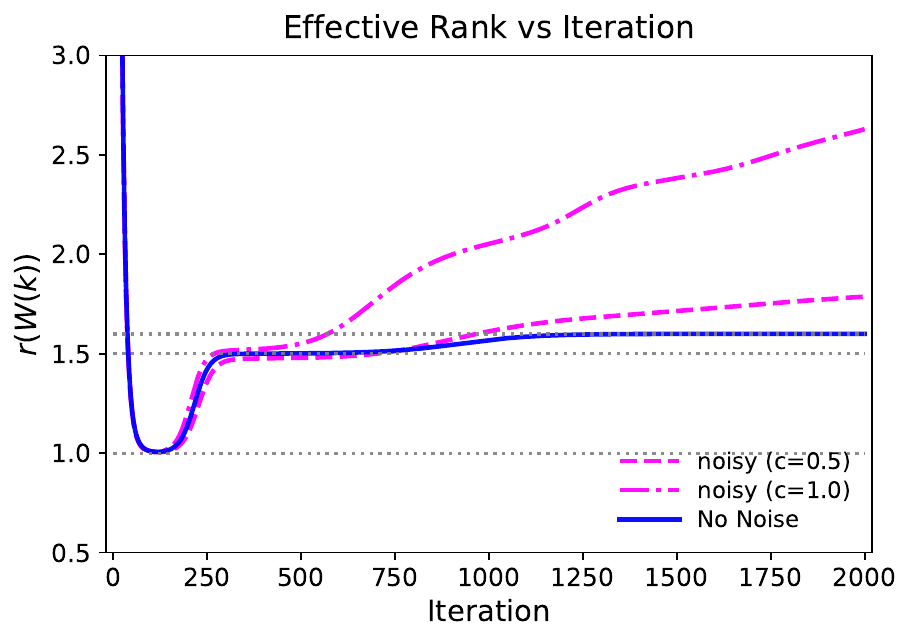}
        \caption{Effective Rank}
        \label{fig:ib-vary-noise}
    \end{subfigure}
 \caption{Robustness of Approximation and Implicit Regularization Induced by \eqref{eq:gd}.
    We set $N = 2$, $n = 20$, $\eta = 0.005$, $\alpha = 10^{-2}$, and $\lambda_1 = 10$, $\lambda_2 = 5$, $\lambda_3 = 1$. Figure~\ref{fig:approx-vary-noise} reports the approximation error as the noise level varies. Figure~\ref{fig:ib-vary-noise} illustrates the corresponding implicit regularization behavior under different noise levels.
  }    
  \label{fig:noise}
\end{figure}
\section{Conclusion}
\label{sec:conclusion}
This paper studied the stability of low-rank implicit regularization in perturbed deep matrix factorization. We first revisited the noiseless setting and derived spectral conditions ensuring the existence of a nonempty low-rank interval. These conditions make explicit how the target spectrum, initialization, and step size govern the emergence of low-rank behavior along the gradient descent trajectory.

We then analyzed the perturbed problem, where the target matrix is subject to an additive perturbation. By studying the perturbed gradient descent dynamics at the eigenvalue level, we established convergence guarantees, quantified the effect of the perturbation size on iteration complexity, and characterized the recovery of nonnegative eigenvalues. Building on these results, we proved that the low-rank phase is stable under controlled perturbations: over a perturbed low-rank interval, the effective rank remains close to that of the rank-$L$ approximation of the noiseless target, with explicit dependence on the perturbation size. We further quantified the shift of the low-rank interval and bounded the resulting low-rank approximation error.

Together, these results show that the low-rank implicit regularization observed in noiseless deep matrix factorization can persist under bounded perturbations of the target matrix. The numerical illustrations support the theoretical predictions and illustrate the role of spectral structure in determining when this stability is observable. Future directions include extending the analysis to nonsymmetric or indefinite settings, deriving sharper probabilistic bounds for concrete random perturbation models, and developing analogous stability results for broader classes of nonlinear models.

\bibliographystyle{elsarticle-num-names}
\bibliography{references}

\end{document}